\documentclass[12pt]{amsart}

\usepackage{amsmath,amsfonts,amssymb}
\usepackage{mathrsfs}
\usepackage{widebar}
\usepackage{fontcommander}
\usepackage{mathtools}
\usepackage{amsthm}
\usepackage[colorlinks=true]{hyperref}
\usepackage{amscd}
\usepackage{enumerate}
\usepackage{float}
\usepackage[dvipsnames]{xcolor}
\usepackage[all,cmtip,2cell]{xy} 
\usepackage{tikz}									
\usetikzlibrary{decorations.pathmorphing,shapes}
\usetikzlibrary{matrix}
\usetikzlibrary{arrows}
\usetikzlibrary{patterns}
\usetikzlibrary{snakes}

\def\overcal#1{\overbar{#1}{\mathcal}}
\makefontcommander{\overcal}{oc}
\makefontcommands ABCDEFGHIJKLMNOPQRSTUVWXYZ \endmakefontcommands

\def\overnorm#1{\overbar{#1}{\mathnormal}}
\makefontcommander{\overnorm}{on}
\makefontcommands ABCDEFGHIJKLMNOPQRSTUVWXYZ \endmakefontcommands
\makefontcommands abcdefghijklmnopqrstuvwxyz \endmakefontcommands

\def\undernorm#1{\underbar{#1}{\mathnormal}}
\makefontcommander{\undernorm}{u}
\makefontcommands ABCDEFGHIJKLMNOPQRSTUVWXYZ \endmakefontcommands

\def\underscr#1{\underbar{#1}{\mathscr}}
\makefontcommander{\underscr}{us}
\makefontcommands ABCDEFGHIJKLMNOPQRSTUVWXYZ \endmakefontcommands

\def\undercal#1{\underbar{#1}{\mathcal}}
\makefontcommander{\undercal}{uc}
\makefontcommands ABCDEFGHIJKLMNOPQRSTUVWXYZ \endmakefontcommands

\makefontcommander{\mathcal}{c}
\makefontcommands ABCDEFGHIJKLMNOPQRSTUVWXYZ \endmakefontcommands

\makefontcommander{\mathscr}{s}
\makefontcommands ABCDEFGHIJKLMNOPQRSTUVWXYZ \endmakefontcommands

\makefontcommander{\widetilde}{t}
\makefontcommands ABCDEFGHIJKLMNOPQRSTUVWXYZ \endmakefontcommands

\def\tildescr#1{\widetilde{\mathscr{#1}}}
\makefontcommander{\tildescr}{ts}
\makefontcommands ABCDEFGHIJKLMNOPQRSTUVWXYZ \endmakefontcommands

\theoremstyle{plain}
\newtheorem{theorem}{Theorem}
\numberwithin{theorem}{section}

\theoremstyle{plain}
\newtheorem{proposition}[theorem]{Proposition}

\theoremstyle{plain}
\newtheorem{lemma}[theorem]{Lemma}

\theoremstyle{plain}
\newtheorem{corollary}[theorem]{Corollary}

\theoremstyle{definition}
\newtheorem{definition}[theorem]{Definition}

\theoremstyle{definition}
\newtheorem{example}[theorem]{Example}

\theoremstyle{remark}
\newtheorem{remark}[theorem]{Remark}

\newcommand{\ChDan}[1]{{#1}}
\newcommand{\ChQile}[1]{{#1}}

\newcommand{\ChMartin}[1]{{#1}}
\newcommand{\ChJonathan}[1]{{#1}}

\renewcommand{\AA}{\mathbb{A}} 
\newcommand{\PP}{\mathbb{P}} 

\newcommand{\CC}{\mathbb{C}} 

\newcommand{\RR}{\mathbb{R}} 
\newcommand{\Gm}{\mathbb{G}_m} 

\newcommand{\ZZ}{\mathbb{Z}} 

\newcommand{\NN}{\mathbb{N}} 

\newcommand\spec{\operatorname{Spec}} 

\DeclareMathOperator{\cok}{coker}

\newcommand{\Log}{{\mathbf{Log}}}
\newcommand{\uLog}{{\underline{\mathbf{Log}}}}

\newcommand{\et}{\mathrm{\acute{e}t}}
\newcommand{\LogSch}{\mathbf{LogSch}} 
\newcommand{\Mon}{\mathbf{Mon}}       
\DeclareMathOperator{\im}{im}
\newcommand{\rest}[1]{\big|_{#1}}
\newcommand{\KF}{\mathbf{KF}}
\newcommand{\AF}{\mathbf{AF}}

\newcommand{\RPC}{\mathbf{RPC}}

\newcommand{\hooklongrightarrow}{\lhook\joinrel\longrightarrow}

\providecommand{\cat}[1]{{\mathsf{#1}}}

\newcommand{\vir}{{\rm vir}}
\renewcommand{\exp}{{\rm exp}}
\newcommand{\R}{\mathbb{R}}
\newcommand{\Rbar}{\overline{\mathbb{R}}}

\newcommand{\A}{\mathbb{A}}
\newcommand{\G}{\mathbb{G}}
\newcommand{\bfp}{\mathbf{p}}
\newcommand{\frakS}{\mathfrak{S}}
\newcommand{\calA}{\mathcal{A}}
\newcommand{\calO}{\mathcal{O}}
\newcommand{\calH}{\mathcal{H}}
\newcommand{\calX}{\mathcal{X}}

\newcommand{\Deltabar}{\overnorm{\Delta}}
\newcommand{\Sigmabar}{\overnorm{\Sigma}}
\newcommand{\sigmabar}{\overnorm{\sigma}}

\newcommand{\Mbar}{\overnorm{M}}

\providecommand{\Kim}{\cat{Kim}}
\providecommand{\Kimstab}{\Kim^{\rm stab}}
\providecommand{\JLi}{\cat{Li}}
\providecommand{\JListab}{\JLi^{\rm stab}}

\providecommand{\ACGS}{\cat{ACGS}}
\providecommand{\ACGSstab}{\ACGS^{\rm stab}}
\providecommand{\Log}{\cat{Log}}
\providecommand{\AFstab}{\cat{AF}^{\rm stab}}

\DeclareMathOperator{\Spec}{Spec}
\DeclareMathOperator{\Hom}{Hom}
\DeclareMathOperator{\trop}{trop}
\DeclareMathOperator{\val}{val}

\DeclareMathOperator{\Span}{Span}

\DeclareMathOperator{\rank}{rank}
\DeclareMathOperator{\id}{id}

\begin{document}

\author[Abramovich, Chen, Marcus, Ulirsch and Wise]{Dan Abramovich, Qile Chen, Steffen Marcus, Martin Ulirsch, and Jonathan Wise}
\date{\today}
\title[Skeletons and fans]{Skeletons and fans of logarithmic structures}

\thanks{Research by Abramovich and Ulirsch is supported in part by NSF grant
   DMS-1162367 and  BSF grant 2010255;
Chen is supported in part by NSF grant DMS-1403271; Wise is supported by an NSA Young Investigator's Grant, Award \#H98230-14-1-0107.}

\address[Abramovich]{Department of Mathematics\\
Brown University\\
Box 1917\\
Providence, RI 02912\\
U.S.A.}
\email{abrmovic@math.brown.edu}

\address[Chen]{Department of Mathematics\\
Columbia University\\
Rm 628, MC 4421\\
2990 Broadway\\
New York, NY 10027\\
U.S.A.}
\email{q\_chen@math.columbia.edu}

\address[Marcus]{Mathematics and Statistics\\
The College of New Jersey\\
Ewing, NJ 08628\\
U.S.A.}
\email{marcuss@tcnj.edu}

\address[Ulirsch]{Department of Mathematics\\
Brown University\\
Box 1917\\
Providence, RI 02912\\
U.S.A.}
\email{ulirsch@math.brown.edu}

\address[Wise]{University of Colorado, Boulder\\
Boulder, Colorado 80309-0395\\ USA}
\email{jonathan.wise@math.colorado.edu}

\maketitle

\setcounter{tocdepth}{1}
\tableofcontents

\section{Introduction}
\label{sec:intro}
\label{Sec:intro}


\subsection{Toric varieties, toroidal embeddings and logarithmic structures} Toric varieties were introduced in  \cite{Demazure} and studied in many sources, see for instance  \cite{KKMS, MO, Danilov, Oda, Fulton_toric, tv}. They are the quintessence of combinatorial algebraic geometry: there is a category of combinatorial objects called fans, which is equivalent to the category of toric varieties with torus-equivariant morphisms between them. Further, classical tropical geometry probes the combinatorial structure of subvarieties of toric varieties.  We  review this theory in utmost brevity in Section~\ref{Sec:toric} below.

In \cite{KKMS} some of this picture was generalized to {\em toroidal embeddings}, especially {\em toroidal embeddings without self intersections}, and their corresponding {\em polyhedral cone complexes} (Section \ref{Sec:toroidal}). Following Kato \cite{Kato-toric} we argue that the correct generality is that of {\em fine and saturated logarithmic structures}. Working over perfect fields, toroidal embeddings can be identified as logarithmically regular varieties. Since logarithmic structures might not be familiar to the reader, we provide a brief review of the necessary definitions in Section \ref{sec:logstr}. To keep matters as simple as possible,  all the logarithmic structures we use below are fine and saturated (Definition \ref{Def:fine}).  These are the logarithmic structures closest to toroidal embeddings which still allow us to pass to arbitrary subschemes.

The purpose of this text is to survey a number of ways one can think of generalizing fans of toric varieties to the realm of logarithmic structures: Kato fans, Artin fans, polyhedral cone complexes, skeletons. These are all closely related and, under appropriate assumptions, equivalent. But the different ways they are realized provide for entirely different applications. We do not address skeletons of logarithmic structures over non-trivially valued nonarchimedean fields---see Werner's contribution \cite{Werner}.
\ChDan{Nor do we address the general theory of skeletons of Berkovich spaces and their generalizations \cite{Berkovich-contraction,Hrushovski-Loeser,MacPherson}.}

\subsection{Kato fans} In \cite{Kato-toric}, Kato associated a combinatorial structure $F_X$, which has since been called {\em the Kato fan of $X$}, to  a logarithmically regular scheme $X$ (with Zariski-local charts), see Section \ref{sec:kato}. This is a reformulation of the polyhedral cone complexes of \cite{KKMS} which realizes them within the category of monoidal spaces. In \cite{Ulirsch_functroplogsch}, one further generalizes the construction of the associated Kato fan to {\em logarithmic structures without monodromy}. As in \cite{KKMS}, Kato fans provide a satisfying theory encoding logarithmically smooth birational modifications in terms of subdivisions of Kato fans. Procedures for resolution of singularities  of polyhedral cone complexes of Kato fans are given in  \cite{KKMS} and \cite{Kato-toric}. As an outcome, we obtain a combinatorial procedure for resolution of singularities of logarithmically smooth structures without self-intersections.

Kato fans, or generalizations of them, can be constructed for more general logarithmic structures.  We briefly discuss such a construction using sheaves on the category of Kato fans.  A different, and possibly more natural approach, is obtained using Artin fans.

\subsection{Artin fans and Olsson's stack of logarithmic structures} In order to import notions and structures from scheme theory to logarithmic geometry, Olsson \cite{Olsson_log} showed that a logarithmic structure $X$ on a given underlying scheme $\uX$ is equivalent to a  morphism $\uX \to \uLog$, where $\uLog$ is a rather large, zero-dimensional   Artin stack  - the moduli stack of logarithmic structures. It carries a universal logarithmic structure whose associated logarithmic algebraic stack we denote by $\Log$ - providing a universal family of logarithmic structures  $\Log \to \uLog$.

Being universal, the logarithmic stack $\Log$ cannot reflect the combinatorics of $X$. In section \ref{sec:artin} we define, following \cite{AW, ACMW}, the notion of {\em Artin fans}, and show that the morphism $X \to \Log$ factors through an initial morphism $X \to \cA_X$, where $\cA_X$ is an Artin fan and  $\cA_X \to \Log$ is \'etale and representable.

We argue that, unlike $\Log$, the Artin fan $\cA_X$ is a combinatorial object which encodes the combinatorial structure of $X$. Indeed, when $X$ is without monodromy, the underlying topological space of $\cA_X  $ is simply the Kato fan $F_X$. In other words, $\cA_X$ combines the advantages of the Kato fan $F_X$, being combinatorial, and of $\Log$, being algebraic. 

In addition $\cA_X$ exists in greater generality, when $X$ is allowed to have self-intersections and monodromy---even not to be logarithmically smooth; in a roundabout way it provides a definition of $F_X$ in this generality, by taking the underlying ``monoidal space''---or more properly, ``monoidal stack''---of $\cA_X$.

The theory of Artin fans is not perfect.  Its current foundations lack full functoriality of the construction of $X \to \cA_X$, just as Olsson's characteristic morphism $X \to \Log$ is functorial only for strict morphisms $Y \to X$ of logarithmic schemes.  In Section~\ref{sec:patch} we provide a patch for this problem, again following Olsson's ideas.

\subsection{Artin fans and unobstructed deformations} Artin fans were developed in \cite{AW, ACMW} and the forthcoming \cite{ACGS} in order to study logarithmic Gromov--Witten theory. The idea is that since an Artin fan $\cA$ is logarithmically \'etale, a map $f: C \to \cA_X$ from a curve to $\cA_X$ is logarithmically unobstructed. Precursors to this result for specific $X$ were obtained in \cite{ACFW,ACW, AMW, CMW}. In \cite{ACFW} an approach to Jun Li's {\em expanded degenerations} was provided using what in hindsight we might call  the Artin fan of the affine line $\cA = \cA_{\AA^1}$. The papers \cite{ACW, AMW, CMW} use this formalism to prove comparison results in relative Gromov--Witten theory. For logarithmically smooth $X$, the map $X \to \cA_X$  was used in \cite{AW} to prove that {\em logarithmic} Gromov--Witten invariants are invariant under logarithmic blowings up; in \cite{ACMW} it was used for general $X$  to complete a proof of boundedness of the space of logarithmic stable maps. 
We review these results, for which both the combinatorial and algebraic features of $\cA_X$ are essential, in Section \ref{sec:obs}.  They serve as evidence that the algebraic structure of Artin fans is an advantage over the purely combinatorial structure of their associated Kato fans.

\subsection{Skeletons and tropicalization} In Section \ref{sec:skel} we follow Thuillier \cite{Thuillier} and associate to a Zariski-logarithmically smooth scheme $X$ its extended cone complex $\Sigmabar_X$.  This is a variant of the cone complex $\Sigma_X$ of \cite{KKMS}, and is related to the Kato fan in an intriguing manner: $$\Sigmabar_X = F_X(\RR_{\geq 0}\sqcup\{\infty\}).$$ The complex $\Sigmabar_X$ is canonically homeomorphic to the skeleton $\frakS(X)$ of the non-archimedean space $X^\beth$ associated to the logarithmic scheme $X$, when viewing the base field as a valued field with trivial valuation, as developed by Thullier \cite{Thuillier}.  In this case there is a continuous map $X^\beth \to \Sigmabar_X$, and $\Sigmabar_X \subset X^\beth$ is a strong deformation retract.  Thuillier used this formalism to prove a compelling result independent of logarithmic or nonarchimedean considerations:  the homotopy type of the dual complex of a logarithmic resolution of a singularites does not depend on the choice of resolution.

For a general fine and saturated logarithmic scheme $X$, we still have a continuous mapping $X^\beth \to \Sigmabar_X$, although we do not have a continuous section $\Sigmabar_X \subset X$. It is argued in \cite{Ulirsch_functroplogsch} that the image of $X^\beth \to \Sigmabar_X$ can be viewed as the {\em tropicalization} of $X$. 

Note that in this discussion we have limited the base field of $X$ to have a trivial absolute value. A truly satisfactory theory must apply to subvarieties defined over valued-field extensions, in particular with non-trivial valuation.

\subsection{Analytification of Artin fans} Artin fans can be tied in to the skeleton picture via their analytifications, as we indicate in Section \ref{sec:anal}. We again consider our base field as a trivially valued field. The analytification of the morphism $\phi_X: X \to \cA_X$ is a morphism $\phi_X^\beth: X^\beth \to \cA_X^\beth$ into an analytic Artin stack $\cA_X^\beth$. The whole structure sits in a commutative diagram
$$\xymatrix{
 X^\beth\ar@/_/[d]_{\rho_X}\ar@/^/[d]^{r_X}\ar[rr]&& \cA_X^\beth\ar@/_/[d]_{\rho_\cA}\ar@/^/[d]^{r_\cA}\ar[rr] &&\Sigmabar_X\ar@/_/[d]_{\rho_\Sigma}\ar@/^/[d]^{r_\Sigma}\\ 
 X \ar[rr] && \cA_X \ar[rr] && F_X
 }$$
On the left hand side of the diagram the horizontal arrows remain in their respective categories - algebraic on the bottom, analytic on top---but discard all geometric data of $X$ and $X^\beth$ except the combinatorics of the logarithmic structure. On the right hand side the arrows  $\cA_X^\beth \to \Sigmabar_X$ and $\cA_X \to F_X$ discard the analytic and algebraic data and preserve topological and monoidal structures. In particular   $\cA_X^\beth \to \Sigmabar_X$  is a {\em homeomorphism}, endowing the familiar complex $\Sigmabar_X$ with an analytic stack structure.

\subsection{Into the future}  While we have provided evidence that the algebraic structure of $\cA_X$ has advantages over the underlying monoidal structure $F_X$, at this point we can only hope that the analytic structure $\cA_X^\beth$ would have \ChDan{significant advantages over the underlying piecewise-linear structure $\Sigmabar_X$, as applications are only starting to emerge, see \cite{Ranganathan-super}.}

We discuss some questions that might usher further applications in Section \ref{sec:future}.

\section{Toric varieties and toroidal embeddings}
\label{Sec:toric}

Mostly for convenience, we work here over an algebraically closed  field $k$.
 We recall, in briefest terms, the standard setup of toric varieties and toroidal embeddings.

\subsection{Toric varieties} Consider a torus $T \simeq \Gm^n$ and a normal variety $X$ on which $T$ acts with a dense orbit isomorphic to $T$---and fix such isomorphism. Write $M$ for the character group of $T$ and write $N$ for the co-character group. Then $M$ and $N$ are both isomorphic to $\ZZ^n$ and are canonically dual to each other. We write $M_\RR = M\otimes\RR$ and $N_\RR = N\otimes\RR$ for the associated real vector spaces.

\subsection{Affine toric varieties and cones} If $X$ is affine it is canonically isomorphic to $X_\sigma = \Spec k[S_\sigma]$, where $\sigma \subset N_\RR$ is a strictly convex rational polyhedral cone, and $S_\sigma = M\cap \sigma^\vee$ is the monoid of lattice points in the $n$-dimensional cone $\sigma^\vee\subset M_\RR$ dual to $\sigma$. Such affine $X_\sigma$ contains a unique closed orbit $\cO_\sigma$, which is itself isomorphic to a suitable quotient torus of $T$.

\subsection{Invariant opens} A nonempty torus-invariant affine open subset of $X_\sigma$ is always of the form $X_\tau$ where $\tau\prec\sigma$ is a face of $\sigma$---either $\sigma$ itself or the intersection of $\sigma$ with a supporting hyperplane.  For instance the torus $T$ itself corresponds to the vertex $\{0\}$ of $\sigma$.

\subsection{Fans} Any toric variety $X$ is covered by invariant  affine opens of the form $X_{\sigma_i}$, and the intersection of $X_{\sigma_i}$ with $X_{\sigma_j}$ are of the form $X_{\tau_{ij}}$ for a common face $\tau_{ij} = \sigma_i \cap \sigma_j$. It follows that the cones $\sigma_i$ form a {\em fan} $\Delta_X$ in $N$. This means precisely that $\Delta_X$ is a collection of strictly convex rational polyhedral cones in $N$, that any face of a cone in $\Delta_X$ is a member of $\Delta_X$, and the intersection of any two cones in $\Delta_X$ is a common face. 

 And vice versa: given a fan $\Delta$ in $N$ one can glue together the associated affine toric varieties $X_\sigma$ along the affine opens $X_\tau$ to form a toric variety $X(\Delta)$.

\begin{figure}[htb]
\begin{center}
\begin{tikzpicture}

\path[pattern=north west lines, pattern color=black!20!white] (1.5,0) -- (1.5,1.5) -- (0,1.5) -- (0,0) -- (1.5,0);
\path[pattern=north east lines, pattern color=black!20!white] (1.5,0) -- (1.5,-3.5) --  (-1,-1) -- (0,0) -- (1.5,0);
\path[pattern=north east lines, pattern color=black!20!white] (0,1.5) -- (-3.5,1.5) -- (-1,-1) -- (0,0) -- (0,1.5);

\fill (0,0) circle (0.020 cm);

\draw (0,0) -- (1,0);
\draw (0,0) -- (0,1);
\draw (0,0) -- (-.75,-.75);

\draw [dotted] (0,1) -- (0,2);
\draw [dotted] (1,0) -- (2,0);
\draw [dotted] (-.75,-.75) -- (-1.5,-1.5);
\end{tikzpicture}
\caption{The fan of $\PP^2$}\label{Fig:SigmaP2}
\end{center}
\end{figure}

\subsection{Categorical equivalence} One defines a morphism from a toric variety $T_1 \subset X_1$ to another $T_2\subset X_2$  to be an equivariant morphism $X_1 \to X_2$ extending a  homomorphism of tori $T_1 \to T_2$. On the other hand one defines a morphism from a fan $\Delta_1$ in $(N_1)_\RR$ to a fan $\Delta_2$ in $(N_2)_\RR$ to be a group homomorphism $N_1 \to N_2$ sending each cone $\sigma\in \Delta_1$ into some cone $\tau \in \Delta_2$.

 A fundamental theorem says
\begin{theorem}
The correspondence above extends to an equivalence of categories between the category of toric varieties over $k$ and the category of fans. 
\end{theorem}

Under this equivalence, toric birational modifications $X_1 \to X_2$ of $X_2$ correspond to subdivisions $\Delta_{X_1} \to \Delta_{X_2}$  of $\Delta_{X_2}$.

\subsection{Extended fans} The closure $\overline \cO_\sigma\subset X$ of a $T$-orbits $\cO_\sigma$, which is itself a toric variety,  is encoded in $\Delta_X$, but in a somewhat cryptic manner. Thuillier \cite{Thuillier} provided a way to add all the fans of these loci  $\overline \cO_\sigma$ and obtain a compactification $\Delta_X\subset \Deltabar_X$: instead of gluing together the cones $\sigma$ along their faces, one replaces $\sigma$ with a natural compactification, the {\em extended  cone} $$\sigmabar:=\Hom_{\Mon}(S_\sigma, \RR_{\geq 0} \cup \{\infty\}),$$  where  the notation $\Hom_{\Mon}$ stands for the set of monoid homomorphisms. This has the effect of adding, in one step, lower dimensional cones isomorphic to $\sigma/\Span{\tau}$ at infinity corresponding to  the closure of $\cO_\tau$ in $X_\sigma$, for all $\tau\prec \sigma$. We still have that  $\overnorm{\tau}_{ij} = \sigmabar_i \cap \sigmabar_j$, and one  can glue together  these extended cones  to obtain the {\em extended fan} $\Deltabar_X$.

\begin{figure}[htb]
\begin{center}
\begin{tikzpicture}
\path[pattern=north west lines, pattern color=black!20!white] (1.5,0) -- (1.5,1.5) -- (0,1.5) -- (0,0) -- (1.5,0);
\path[pattern=north east lines, pattern color=black!20!white] (1.5,0) -- (1.5,-3.5) --  (-1,-1) -- (0,0) -- (1.5,0);
\path[pattern=north east lines, pattern color=black!20!white] (0,1.5) -- (-3.5,1.5) -- (-1,-1) -- (0,0) -- (0,1.5);

\fill (2,2) circle (0.05 cm);
\fill (0,0) circle (0.020 cm);
\fill (0,2) circle (0.03 cm);
\fill (-1.5,-1.5) circle (0.03 cm);
\fill (-5,2) circle (0.05 cm);
\fill (2,-5) circle (0.05 cm);
\fill (2,0) circle (0.03 cm);

\draw (0,0) -- (1,0);
\draw (0,0) -- (0,1);
\draw (2,0) -- (2,1);
\draw (0,2) -- (1,2);
\draw (0,0) -- (-.75,-.75);
\draw (0,2) -- (-4,2);
\draw (2,0) -- (2,-4);
\draw (-1.5,-1.5) -- (1,-4);
\draw (-1.5,-1.5) -- (-4,1);

\draw [dotted] (0,1) -- (0,2);
\draw [dotted] (1,0) -- (2,0);
\draw [dotted] (2,1) -- (2,2);
\draw [dotted] (1,2) -- (2,2);
\draw [dotted] (-.75,-.75) -- (-1.5,-1.5);
\draw [dotted] (-4,2) -- (-5,2);
\draw [dotted] (-4,1) -- (-5,2);
\draw [dotted] (1,-4) -- (2,-5);
\draw [dotted] (2,-4) -- (2,-5);

\end{tikzpicture}
\caption{The extended fan of $\PP^2$}\label{Fig:SigmabarP2}
\end{center}
\end{figure}


\subsection{Toroidal embeddings}\label{Sec:toroidal} The theory of toroidal embedding was developed in \cite{KKMS} in order to describe varieties that look locally like toric varieties. A toroidal embedding $U\subset X$ is a dense open subset of a normal variety $X$ such that, for any closed point $x$, the completion $\widehat U_x \subset \widehat X_x$ is isomorphic to the completion  of an affine toric variety $T_x\subset X_{\sigma_x}$. Equivalently, each $x\in X$ should admit an \'etale neighborhood $\phi_x:V_x \to X$ and an \'etale morphism $\psi_x:V_x \to X_{\sigma_x}$ such that $\psi_x^{-1} T_x = \phi_x^{-1} U$. Note that the open set $U\subset X$ serves as a global structure connecting the local pictures $T_x\subset X_{\sigma_x}$.

\subsection{The cone complex of a toroidal embedding without self-intersections}\label{Sec:the-cone-complex} If the morphisms $\phi_x:V_x \to X$  are assumed to be Zariski open embeddings, then the toroidal embedding is a {\em toroidal embedding without self-intersections}. In this case the book  \cite{KKMS}  provides a polyhedral cone complex $\Sigma_X$ replacing the fan of a toric \ChDan{variety.} The main difference is that the cones of the complex $\Sigma_X$ do not lie linearly  inside an ambient space of the form \ChDan{$N_\RR$.}

For a toroidal embedding without self-intersections, the strata of $X_{\sigma_x}$ glue together to form a stratification $\{\cO_i\}$ of $X$. For $x\in \cO_i$   the cone $\sigma_x$, along with its sublattice $\sigma_x \cap N$, is independent of   $x$. It can be described canonically as \ChDan{follows. Let} $X_i$ be the {\em star} of $\cO_i$, namely the union of strata containing $\cO_i$ in their \ChDan{closures.  It} is an open subset of $X$.  Let $\Mbar_i$ be  the monoid of effective Cartier divisors on $X_i$ supported on $X_i\smallsetminus U$. Let $N_\sigma  = \Hom_{\Mon}(\Mbar_i, \NN)$ be the dual monoid. Then $\sigma_x = (N_\sigma)_\RR$ is the associated cone. When one passes to another stratum contained in $X_i$ one obtains a face $\tau\prec \sigma$, and $N_\tau = \tau\cap N_\sigma$. These cones glue together naturally, in a manner compatible with the sublattices, to form a cone complex $\Sigma_X$ with integral structure. 
Unlike the case of fans, the intersection  $\sigma_i \cap\sigma_j$ could be a whole common sub-fan of $\sigma_i$ and $\sigma_j$, and not necessarily one cone.

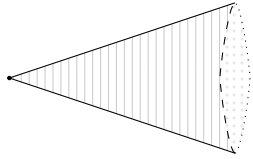
\begin{figure}[htb]
\begin{center}
\begin{tikzpicture}

\path [pattern=vertical lines, pattern color=black!20!white] (-3,1) -- (0,0) .. controls +(160:0.1cm) and +(-90:0.1cm) .. (-0.2,1) -- (-3,1);
\path [pattern=vertical lines, pattern color=black!20!white] (-3,1) -- (-0.2,1) .. controls +(90:0.1cm) and +(200:0.1cm) .. (0,2) -- (-3,1);

\draw (-3,1) -- (0,0);
\draw (-3,1) -- (0,2);

\fill (-3,1) circle (0.030cm);


\draw [dotted, pattern=dots, pattern color=black!10!white] (0,0) .. controls +(20:0.1cm) and +(-90:0.1cm) .. (0.2,1) .. controls +(90:0.1cm) and +(-20:0.1cm) .. (0,2);
\draw [dashed, pattern=dots, pattern color=black!10!white](0,0) .. controls +(160:0.1cm) and +(-90:0.1cm) .. (-0.2,1) .. controls +(90:0.1cm) and +(200:0.1cm) .. (0,2);

\end{tikzpicture}
\caption{The complex of $\PP^2$ with the divisor consisting of a line and a transverse conic: the cones associated to the zero strata meet along both their edges.}\label{Fig:Complexnotfan}
\end{center}
\end{figure}


\subsection{Extended complexes} Just as in the case of toric varieties, these complexes canonically admit compactifications $\Sigma_X \subset \Sigmabar_X$, obtained by replacing each cone $\sigma_x$ by the associated extended cone $\sigmabar_x$. This structure was introduced in Thuillier's \cite{Thuillier}.

\subsection{Functoriality} 

Let $U_1\subset X_1$ and $U_2\subset X_2$ be toroidal embeddings without self-intersections and $f: X_1 \to X_2$ a dominant morphism  such that $f(U_1) \subset U_2$. Then one canonically obtains a mapping $\Sigma_{X_1} \to \Sigma_{X_2}$, simply because Cartier divisors supported away from $U_2$ pull back to Cartier divisors supported away from $U_1$. This mapping is continuous, sends cones into cones linearly, and sends lattice points to lattice points. Declaring such mappings to be mappings of polyhedral cone complexes with integral structure, we obtain a functor from toroidal embeddings to polyhedral cone complexes. This functor is far from being an equivalence. 

In \cite{KKMS} one focuses on {\em toroidal modifications}   $f: X_1 \to X_2$, namely those birational modifications described on charts of $X_2$ by toric modifications of the toric varieties $X_{\sigma_x}$. Then one shows

\begin{theorem}
The correspondence $X_1 \mapsto \Sigma_{X_1}$ extends to an equivalence of categories between toroidal modifications of $U_2\subset X_2$, and subdivisions of $\Sigma_{X_2}$. 
\end{theorem}

\section{Logarithmic structures}
\label{sec:logstr}
We briefly review the theory of logarithmic structures \ChDan{\cite{Kato}}. 

\subsection{Notation for monoids}

\begin{definition}\label{monoid}
A \emph{monoid} is a commutative semi-group with a unit. A morphism of monoids is required to preserve the unit element. 

We denote the category of monoids by the symbol $\Mon$.
\end{definition}

Given a monoid $P$, we can associate a group 
\begin{equation*}{
P^{gp}:=\{(a,b) |(a,b)\sim(c,d) \ \mbox{if} \ \exists s \in
P \ \mbox{such that} \ s+a+d=s+b+c\}.} 
\end{equation*}
 Note that any morphism from $P$ to an abelian group factors through $P^{gp}$ uniquely.

\begin{definition}\label{Def:integral} 
A monoid $P$ is called \emph{integral} if $P\rightarrow P^{gp}$ is injective. It is called \emph{fine} if it is integral and finitely generated.  

An integral monoid $P$ is said to be \emph{saturated} if whenever $p\in P^{gp}$ and $n$ is a positive integer such that $np \in P$ then $p \in P$.

As has become customary, we abbreviate the combined condition ``fine and saturated"  to \emph{fs}. 
\end{definition}

\subsection{Logarithmic structures}

\begin{definition}\label{log-str}
Let $\uX$ be a scheme. A \emph{pre-logarithmic structure} on $\uX$ is a sheaf of
monoids $M_{X}$ on the small \'etale site $\et(\uX)$ combined with a
morphism of sheaves of monoids: $\alpha
: M_{X} \longrightarrow \mathcal{O}_{\uX}$, called the \emph{structure morphism}, where
we view $\mathcal{O}_{\uX}$ as a monoid under multiplication. A pre-logarithmic
structure is called a \emph{logarithmic structure} if
$\alpha^{-1}(\mathcal{O}_{\uX}^{*})\cong \mathcal{O}^{*}_{\uX}$ via $\alpha$. The
pair $(\uX,M_{X})$ is called a \emph{logarithmic scheme}, and will be denoted by $X$.  

The structure morphism $\alpha$ is frequently denoted $\exp$ and an inverse $\mathcal{O}_X^\ast \rightarrow M_X$ is denoted $\log$.  
\end{definition}

\begin{definition}
\label{Def:chara}
Given a logarithmic scheme $X$, the quotient sheaf $\onM_{X}=M_{X}/\cO_{\uX}^{*}$ is called the \emph{characteristic monoid}, or just the \emph{characteristic}, of the logarithmic structure $M_{X}$.
\end{definition}

\begin{definition} 
Let  $M$ and $N$ be pre-logarithmic structures on $\uX$. A {\em morphism}
between them is a morphism $M \rightarrow N$ of sheaves of monoids
which is compatible with the structure morphisms. 
\end{definition}

\begin{definition}
Let  $\alpha : M \rightarrow \cO_{\uX}$ be a pre-logarithmic structure on $\uX$. We define the  \emph{associated  logarithmic structure} $M^{a}$ to be the push-out of 
\[
\xymatrix{
\alpha^{-1}(\cO_{\uX}^{*}) \ar[d] \ar[r] & M \\
\cO_{\uX}^{*}
}
\]
in the category of sheaves of monoids on $\et(\uX)$, endowed with 
\[M^{a} \rightarrow \cO_{\uX}  \qquad (a,b)\mapsto \alpha(a)b \qquad\qquad (a \in M, b \in \cO_{\uX}^{*}).\]
\end{definition}

The following are two standard examples from \cite[(1.5)]{Kato}:

\begin{example}\label{Ex:divisorial-log-structure}
Let $\uX$ be a smooth scheme with an effective  divisor $D \subset \uX$. Then we have a standard logarithmic structure $M$ on $\uX$ associated to the pair $(\uX, D)$, where
\[
M_X := \{f \in \cO_{\uX} \ | \ f|_{\uX\setminus D} \in \cO^*_{\uX}\}
\]
with the structure morphism $M_{X} \to \cO_{\uX}$ given by the canonical inclusion. \ChDan{This is already a logarithmic structure, as any section of $\cO^*_{\uX}$ is already in $M_X$.}
\end{example}

\begin{example}\label{ex:affine-toric-log}
Let $P$ be an \emph{fs} monoid, and $\uX = \spec \ZZ[P]$ be the associated affine toric scheme. Then we have a standard logarithmic structure $M_{X}$ on $\uX$ associated to the pre-logarithmic structure
\[
P \to \ZZ[P]
\]
defined by the obvious inclusion. 

We denote by $\spec(P\to\ZZ[P])$ the log scheme $(\uX,M_{X})$.
\end{example}

\subsection{Inverse images}

Let $f:\uX \rightarrow \uY$ be a morphism of schemes. Given a logarithmic structure
$M_{Y}$ on $\uY$, we can define a logarithmic structure on $\uX$, called the
\emph{inverse image} of $M_{Y}$, to be the logarithmic structure associated to the
pre-logarithmic structure $f^{-1}(M_{Y})\rightarrow
f^{-1}(\cO_{\uY})\rightarrow \cO_{\uX}$. This is usually denoted by
$f^{*}(M_{Y})$. Using the inverse image of logarithmic structures, we can
give the following definition. 

\begin{definition}\label{Def:log-mor} 
A \emph{morphism of logarithmic schemes} $X\rightarrow Y$
consists of a morphism of underlying schemes $f:\uX\rightarrow \uY$, and a
morphism  $f^{\flat}: f^{*}M_{Y}\rightarrow M_{X}$ of logarithmic
structures on $\uX$.  
The morphism is said to be \emph{strict} if $f^\flat$ is an isomorphism.

We denote by $\LogSch$ the category of 
logarithmic schemes.  
\end{definition}

\subsection{Charts of logarithmic structures}

\begin{definition}\label{Def:chart} 
Let $X$ be a logarithmic scheme, and $P$ a monoid. A {\em chart} for $M_{X}$
is a morphism $P\rightarrow \Gamma(X,M_{X})$, such that the induced
map of logarithmic strucutres $P^{a}\to M_{X}$ is an isomorphism, where
$P^{a}$ is the logarithmic structure associated to the pre-logarithmic structure given
by
$P\rightarrow \Gamma(X,M_{X})\rightarrow \Gamma(X,\cO_{\uX})$.  
\end{definition}

In fact, a chart of $M_{X}$ is equivalent to a morphism $$f:X \rightarrow \Spec(P\rightarrow\ZZ[P])$$%
such that $f^{\flat}$ is an isomorphism. In general, we have the following:

\begin{lemma}\cite[1.1.9]{Ogus} The mapping
$$\Hom_{\LogSch}(X,\Spec(P\rightarrow \ZZ[P]))\ \to \ \Hom_{\Mon}(P,\Gamma(X,M_{X}))$$
associating to $f$ the composition
 $$\xymatrix{P \ar[r]&\Gamma(\uX, P_X)\ar[r]^{\Gamma(f^\flat)} & \Gamma(\uX, M_X)}$$
 is a bijection.
 \end{lemma}

We will see in Example~\ref{ex:artin-cones} that Artin cones have a similar universal property on the level of characteristic monoids.

\begin{definition}\label{Def:fine} 
A logarithmic scheme $X$ is said to be \emph{fine}, if \'etale locally there is a chart $P\rightarrow M_{X}$ with $P$ a fine monoid. If moreover $P$ can be chosen to be saturated, then $X$ is called a \emph{fine and saturated} (or \emph{fs}) logarithmic structure.  Finally, if $P$ can be chosen isomorphic to $\NN^k$ we say that the logarithmic structure is \emph{locally free}.
\end{definition}



\begin{lemma}
Let $X$ be a fine and saturated logarithmic scheme. Then for any geometric point $\bar{x} \in X$, there exists an \'etale neighborhood $U \to X$ of $\bar{x}$ with a chart $\onM_{\bar{x},X} \to M_{U}$, such that the composition $\onM_{\bar{x},X} \to M_{U} \to \onM_{\bar{x},X}$ is the identity.
\end{lemma}
\begin{proof}
This is a special case of \cite[Proposition 2.1]{Olsson_log}.
\end{proof}

\subsection{Logarithmic differentials}

To form sheaves of logarithmic differentials, we add to the sheaf $\Omega_{\uX/\uY}$ symbols of the form $d\log(\alpha(m))$ for all elements $m\in M_X$, as follows:

\begin{definition}\label{log-differential}
Let $f: X \to Y$ be a morphism of fine logarithmic schemes. We introduce the sheaf of relative logarithmic differentials $\Omega_{X/Y}^{1}$ given by  \[\Omega_{X/Y}^{1}=\Bigl( \Omega_{\uX/\uY}\oplus(\cO_{\uX}\otimes_{\ZZ} M^{gp}_{X}) \Bigr)\big/\mathcal{K}\]
where $\mathcal{K}$ is the $\cO_{\uX}$-module generated by local 
sections of the following forms:
 \begin{enumerate}
   \item $(d\alpha(a),0)-(0,\alpha(a)\otimes a)$ with $a\in M_{X}$;
   \item $(0,1\otimes a)$ with $a\in \im(f^{-1}(M_{Y})\rightarrow M_{X})$.
 \end{enumerate}
The universal derivation $(\partial,D)$ is given by $\partial : \cO_{\uX}\stackrel{d}\rightarrow\Omega_{\uX/\uY}\rightarrow\Omega_{X/Y}^{1}$ and $D: M_{X}\rightarrow \cO_{\uX}\otimes_{\ZZ} M_{X}^{gp}\rightarrow\Omega_{X/Y}^{1}$.
\end{definition}

\begin{example}
Let $h: Q \rightarrow P$ be a morphism of fine monoids. Denote $X=\spec(P\rightarrow\ZZ[P])$ and $Y=\spec(Q\rightarrow\ZZ[Q])$. Then we have a morphism $f: X\rightarrow Y$ induced by $h$. A direct calculation shows that $\Omega_{f}^{1}=\cO_{\uX}\otimes_{\ZZ} \cok (h^{gp})$. The free generators correspond to the logarithmic differentials $d\log(\alpha(p))$ for $p\in P$, which are regular on the torus $\spec\ZZ[P^{gp}]$, modulo those coming from $Q$.   This can also be seen from the universal property of the sheaf of logarithmic differentials.
\end{example}

\subsection{Logarithmic smoothness}

Consider  the following commutative diagram of logarithmic schemes illustrated with solid arrows:

\begin{equation}\label{diag:smooth}
\vcenter{\xymatrix{
T_{0} \ar[r]^{\phi}  \ar[d]^{j}_{J}& X \ar[d]^{f}\\
T_{1} \ar[r]_{\psi} \ar@{-->}[ur] & Y
}}
\end{equation}
where $j$ is a {\em strict} closed immersion (Definition~\ref{Def:log-mor}) defined by the ideal $J$ with $J^{2}=0$. We  define logarithmic smoothness by the infinitesimal lifting property:

\begin{definition}\label{defn:logsmooth}
A morphism $f:X\rightarrow Y$ of fine logarithmic schemes is called \emph{logarithmically smooth} (resp.\ \emph{logarithmically \'etale}) if the underlying morphism $\uX\rightarrow \uY$ is locally of finite presentation and for any commutative diagram (\ref{diag:smooth}), \'etale locally on $T_{1}$  there exists a (resp.\ there exists a unique) morphism $g:T_{1}\rightarrow X$ such that $\phi=g\circ j$ and $\psi=f\circ g$.
\end{definition}

 We have the following useful criterion for smoothness from \cite[Theorem 3.5]{Kato}.

\begin{theorem}[K.\ Kato]\label{KatoStrThm} 
Let $f:X\rightarrow Y$ be a morphism of fine logarithmic schemes. Assume we have a chart $Q\rightarrow M_{Y}$, where $Q$ is a finitely generated integral monoid. Then the following are equivalent:
 \begin{enumerate}
  \item $f$ is logarithmically smooth (resp. logarithmically \'etale);
  \item \'etale locally on $X$, there exists a chart $(P_{X}\rightarrow M_{X},Q_{Y}\rightarrow M_{Y},Q\rightarrow P)$ extending the chart $Q_{Y}\rightarrow M_{Y}$, satisfying the following properties.
    \begin{enumerate}
      \item The kernel and the torsion part of the cokernel (resp. the kernel and the cokernel) of $Q^{gp}\rightarrow P^{gp}$ are finite groups of orders invertible on $X$.
      \item The induced morphism from $\uX\rightarrow \uY\times_{\spec\ZZ[Q]} \spec\ZZ[P]$ is \'etale
in the usual sense.
    \end{enumerate}
 \end{enumerate}
\end{theorem}

\begin{remark}
 \begin{enumerate}
   \item We can require $Q^{gp}\rightarrow P^{gp}$ in (a) to be injective, and replace the requirement that  $\uX\rightarrow \uY\times_{\spec\ZZ[Q]} \spec\ZZ[P]$
be \'etale in (b) by requiring it to be smooth without changing the conclusion of Theorem \ref{KatoStrThm}. 
   \item In this theorem something wonderful happens, which Kato calls ``the magic of log".  The arrow in (b) shows that a logarithmically smooth morphism is ``locally toric'' relative to the base. Consider the case $Y$ is a logarithmic scheme with underlying space given by $\spec\CC$ with the trivial logarithmic structure, and $X=\spec(P\rightarrow \CC[P])$ where $P$ is a fine, saturated and torsion free monoid. Then $\uX$ is a toric variety with the action of $\spec\CC[P^{gp}]$. According to the theorem, $X$ is logarithmically smooth relative to $Y$, though the underlying space might be singular. These singularities are called toric singularities in \cite{Kato-toric}. This is closely related to the classical notion of toroidal embeddings~\cite{KKMS}.
 \end{enumerate}
\end{remark}

Logarithmic differentials behave somewhat analogously to differentials:

\begin{proposition}\label{prop:logsmoothdif} 
Let $X\stackrel{f}{\rightarrow}Y\stackrel{g}{\rightarrow}Z$ be a sequence of morphisms of fine logarithmic schemes. 
 \begin{enumerate}
  \item There is a \ChDan{natural sequence} $f^{*}\Omega_{g}^{1} \rightarrow \Omega_{g f}^{1}\rightarrow \Omega_{f}^{1}\rightarrow 0$ is exact.
  \item If $f$ is logarithmically smooth, then $\Omega_{f}^{1}$ is a locally free $\cO_{X}$-module, and we have the following exact sequence: $0\rightarrow f^{*}\Omega_{g}^{1} \rightarrow \Omega_{g f}^{1}\rightarrow \Omega_{f}^{1}\rightarrow 0$.
  \item If $g f$ is logarithmically smooth and the sequence in (2) is exact and splits locally, then $f$ is logarithmically smooth.
 \end{enumerate}
\end{proposition}

A proof can be found in \cite[Chapter IV]{Ogus}.

\section{Kato fans and resolution of singularities}
\label{sec:kato}

\subsection{The monoidal analogues of schemes}

In parallel to the theory of schemes, Kato developed a theory of fans, with the role of commutative rings played by monoids.  As with schemes, the theory begins with the spectrum of a monoid:

\begin{definition}[{\cite[Definition~(5.1)]{Kato-toric}}]
Let $M$ be a monoid.  A subset $I \subset M$ is called an \emph{ideal} of $M$ if $M + I \subset I$.  If $M \smallsetminus I$ is a submonoid of $M$ then $I$ is called a \emph{prime} ideal of $M$.  The set of prime ideals of $M$ is denoted $\Spec M$ and called the \emph{spectrum} of $M$.
\end{definition}

If $f : M \rightarrow N$ is a homomorphism of monoids and $P \subset N$ is a prime ideal then $f^{-1} P \subset M$ is a prime ideal as well.  Therefore $f$ induces a morphism of spectra:  $\Spec N \rightarrow \Spec M$.

\begin{definition}[{\cite[Definition~(5.2)]{Kato-toric}}]
Suppose that $M$ is a monoid and $S$ is a subset of $M$.  We write $M[-S]$ for the initial object among the monoids $N$ equipped a morphism $f : M \rightarrow N$ such that $f(S)$ is invertible.  When $S$ consists of a single element $s$, we also write $M[-s]$ in lieu of $M[-\{ s \}]$.
\end{definition}

It is not difficult to construct $M[-S]$ with the familiar Grothendieck group construction $M \mapsto M^{gp}$ of Definition \ref{monoid}.
Certainly $M[-S]$ coincides with $M[-S']$ where $S'$ is the submonoid of $M$ generated by $S$.  One may therefore assume that $S$ is a submonoid of $M$.  Then for $M[-S]$ one may take the set of formal differences $m - s$ with $m \in M$ and $s \in S$, subject to the familiar equivalence relation:
\begin{equation*}
m - s \sim m' - s' \qquad \iff \qquad \exists \: t \in S, \: t + m + s' = t + m' + s
\end{equation*}
If $M$ is integral then one may construct $M[-S]$ as a submonoid of $M^{\rm gp}$.  

The topology of the spectrum of a monoid is defined exactly as for schemes:

\begin{definition}[{\cite[Definition~(9.2)]{Kato-toric}}]
Let $M$ be a monoid.  For any $f \in M$, let $D(f) \subset \Spec M$ be the set of prime ideals $P \subset M$ such that $f \not\in P$.  A subset of $\Spec M$ is called open if it is open in the minimal topology in which the $D(f)$ are open subsets.
\end{definition}

Equivalently $D(f)$ is the image of the map $\Spec(M[-f]) \rightarrow \Spec M$.  The intersection of $D(f)$ and $D(g)$ is $D(f + g)$ so the sets $D(f)$ form a basis for the topology of $\Spec M$.

\begin{figure}[htb]
\begin{center}

%
\includegraphics[scale=0.6]{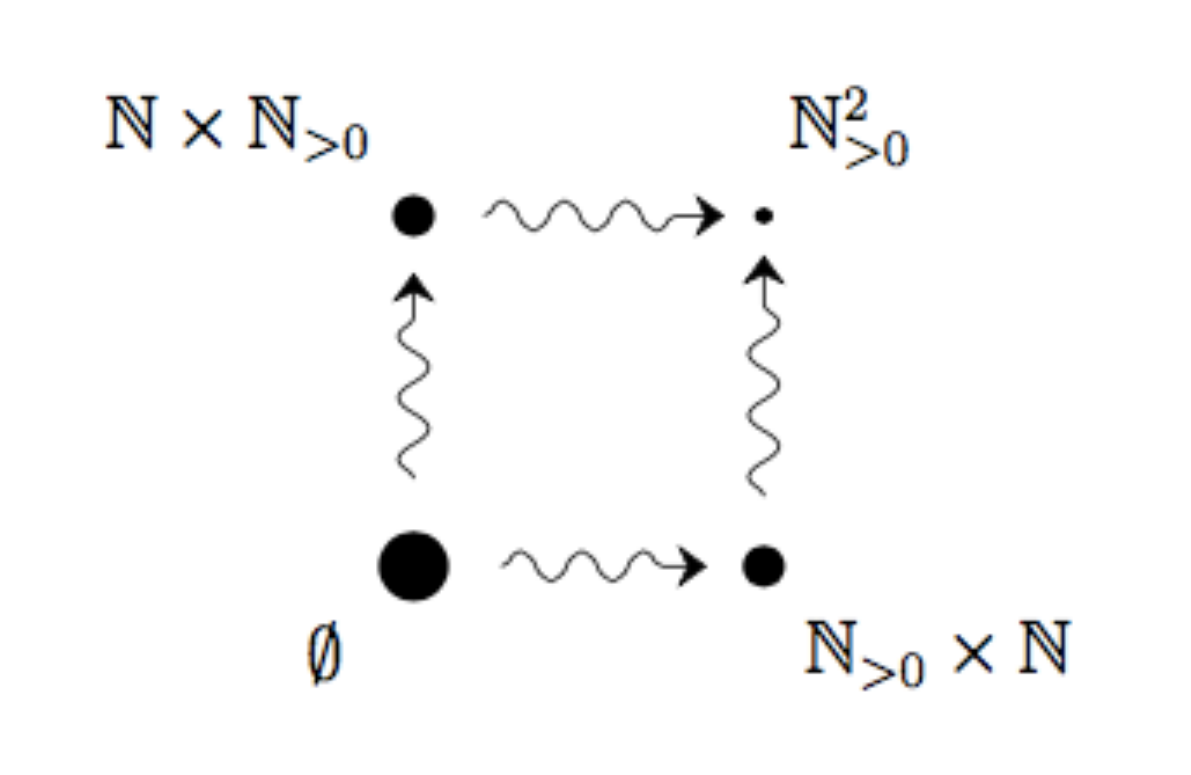}
\caption{\ChDan{Points and topology of $\Spec \NN^2.$} The big point corresponds to the ideal $\varnothing$, the intermediate points are the primes $\NN \times \NN_{>0}$ and $\NN_{>0} \times \NN$, and the small closed point is the maximal ideal $\NN_{>0}^2$. {\ChDan{The arrow indicate specialization, thus determining the topology.}}}\label{Fig:SpecN2}
\end{center}
\end{figure}

We equip $\Spec M$ with a sheaf of monoids $\cM_{\Spec M}$ where 
\begin{equation*}
\cM_{\Spec M}(D(f)) = M[-f] / M[-f]^\ast
\end{equation*}
where $M[-f]^\ast$ is the set of invertible elements of $M$.  This is a sharply monoidal space:

\begin{definition}[{\cite[Definition~(9.1)]{Kato-toric}}]
Recall that a monoid is called \emph{sharp} if its only invertible element is the identity \ChJonathan{element} $0$.  If $M$ and $N$ are sharp monoids then a sharp homomorphism $f : M \rightarrow N$ is a homomorphism of monoids such that $f^{-1} \{ 0 \} = \{ 0 \}$.

A \emph{sharply monoidal space} is a pair $(S, \cM_S)$, where $S$ is a topological space and $\cM_S$ is a sheaf of sharp monoids on $S$.  A morphism of sharply monoidal spaces $f : (S, \cM_S) \rightarrow (T, \cM_T)$ consists of a continuous function $f : S \rightarrow T$ and a sharp homomorphism of sheaves of sharp monoids $f^{-1} \cM_T \rightarrow \cM_S$.%

A sharply monoidal space is called an \emph{affine Kato fan} or a \emph{Kato cone} if it is isomorphic to $(\Spec M, \mathcal{M}_{\Spec M})$.  A sharply monoidal space is called a \emph{Kato fan} if it admits an open cover by  Kato cones.  

We call a Kato fan \emph{integral} or \emph{saturated} if it admits a cover by the spectra of monoids with the respective properties (Definition~\ref{Def:integral}).  A Kato fan is called \emph{locally of finite type} if it admits a cover by spectra of finitely generated monoids.  We use \emph{fine} as a synonym for integral and locally of finite type.
\end{definition}

\ChDan{A large collection of examples of Kato fans is obtained from fans of toric varieties (see Section \ref{Sec:fan-to-kato-fan}) or logarithmically smooth schemes (Section \ref{Sec:associated-Kato}). In particular, any toric \ChMartin{singularity}  is manifested in a Kato fan. 

Any Kato cone $\Spec M$ of a finitely generated integral monoid $M$  contains an open point corresponding to the ideal $\varnothing\subset M$, carrying the trivial stalk $M^{\operatorname{gp}}/M^{\operatorname{gp}} = 0$. We can always glue an arbitrary collection of  such Kato cones along their open points. If the collection is infinite this gives examples of  connected  Kato fans which are not quasi-compact. 
}

\subsection{Points and Kato cones}

\begin{definition}
Let $F' \rightarrow F$ be a morphism of fine, saturated Kato fans.  The morphism is said to be \emph{quasi-compact} if the preimage of any open subcone of $F$ is quasi-compact.
\end{definition}

\begin{lemma}
Let $F$ be a Kato fan.  There is a bijection between the open Kato subcones of $F$ and the points of $F$.
\end{lemma}
\begin{proof}
Suppose that $U = \Spec M$ is an open subcone of $F$.  Let $P \subset M$ be the complement of $0 \in M$.  Then $P$ is a prime ideal, hence corresponds to a point of $F$.  

To give the inverse, we show that every point of $F$ has a minimal open affine neighborhood.  Indeed, suppose that $U$ is an open affine neighborhood of $p \in F$.  Then the underlying topological space of $U$ is finite, so there is a smallest open subset of $U$ containing $p$.  Replace $U$ with this open subset.  It must be affine, since affine open subsets form a basis for the topology of $U$.  

It is straightforward to see that these constructions are inverse to one another.
\end{proof}

\begin{lemma}
A morphism of fine, saturated Kato fans is quasi-compact if and only if it has finite fibers.  In particular, a Kato fan is quasi-compact if and only if its underlying set is finite.
\end{lemma}

\subsection{From fans to Kato fans}\label{Sec:fan-to-kato-fan} If $\Delta$ is a fan in $N_\RR$ in the sense of toric geometry, it gives rise to a Kato fan $F_\Delta$.  The underlying topological space of $F_\Delta$ is the set of cones of $\Delta$ and a subset is open if and only if it contains all the faces of its elements.  In particular, if $\sigma$ is one of the cones of $\Delta$ then the set of faces of $\sigma$ is an open subset $F_\sigma$ of $K$ and these open subsets form a basis.  We set $\mathcal{M}(F_\sigma) = (M \cap \sigma^\vee) / (M \cap \sigma^\vee)^\ast$ where $M$ is the dual lattice of $N$ and $\sigma^\vee$ is the dual cone of $\sigma$.  With this sheaf of monoids, $F_\sigma \simeq \Spec (M \cap \sigma^\vee) / (M \cap \sigma^\vee)^\ast$, so $F_\Delta$ has an open cover by Kato cones, hence is a Kato fan.

\subsection{Resolution of singularities of Kato fans}

Kato introduced the following monoidal space analogue of subdivisions of  fans of toric varieties, in such a way that  a subdivision of a fan $\Sigma$ gives rise  to a subdivision of the Kato fan $F_\Sigma$. A morphism of fans $\Sigma_1 \to \Sigma_2$ is a subdivision if and only if it induces a bijection on the set of lattice points $\cup_\sigma N_\sigma$.  The Kato fan notion is the direct analogue:

\begin{definition} \label{def:proper-subdivision}
A morphism $p : F' \rightarrow F$ of fine, saturated Kato fans is called a \emph{proper subdivision} if it is quasi-compact and the morphism
\begin{equation*}
\Hom(\Spec \NN, F') \rightarrow \Hom(\Spec \NN, F)
\end{equation*}
is a bijection.
\end{definition}

\begin{remark}
This definition has an appealing resemblance to the valuative criterion for properness.
\end{remark}

Explicitly subdividing Kato fans is necessarily less intuitive than subdividing fans. The following examples, which are in direct analogy to subdivisions of fans,  may help in developing intuition:

\begin{example}\label{ex:star}
\begin{enumerate}[(i)]
\item Suppose that $X$ is a fine, saturated Kato fan and $v : \Spec \NN \rightarrow X$ is a morphism.  The \emph{star subdivision} of $X$ along $v$ is constructed as follows:  If $U = \Spec M$ is an open Kato subcone of $X$ not containing $v$ then $U$ is included as an open Kato subcone of $X'$; if $U$ does contain $v$ then for each face $V = \Spec N$ contained in $U$ that does not contain $v$, we include the face $v + V=\Spec M_{v+V}$, where  $$M_{v+V}  =  \left\{\alpha \in M^{\rm gp}\ :\ \alpha \rest{V} \in N \text{ and } v^\ast \alpha \in \NN\right\}.$$
\item Let $X = \Spec M$ be a fine, saturated Kato cone.  There is a canonical morphism $\Spec \NN \rightarrow X$ by regarding $\Hom(\Spec \NN, X) = \Hom(M, \NN)$ as a monoid and taking the sum of the generators of the $1$-dimensional faces of $X$.  This morphism is called the \emph{barycenter} of $X$.

If $X$ is a fine, saturated Kato fan we obtain a subdivision $X' \rightarrow X$ by performing star subdivision of $X$ along the barycenters of its open subcones, in decreasing order of dimension.  A priori this is well defined if cones of $X$ have bounded dimension, but as the procedure is compatible with restriction to subfans, this works for arbitrary $X$. This subdivision is called the \emph{barycentric subdivision}.
\end{enumerate}
\end{example}

\begin{definition} \label{def:smoothness}
A fine, saturated Kato fan $X$ is said to be \emph{smooth} if it has an open cover by Kato cones $U \simeq \Spec \NN^r$.
\end{definition}

\begin{theorem}[{\cite[Theorem~I.11]{KKMS}, \cite[Section~2.6]{Fulton_toric}, \cite[Proposition~(9.8)]{Kato-toric}}]
Let $X$ be a fine, saturated Kato fan.  Then there is a proper subdivision $X' \rightarrow X$ such that $X'$ is smooth.
\end{theorem}

The classical combinatorial proofs start by first using star or barycentric  subdivisions to make the fan simplicial, and then repeatedly reducing the index by further star subdivisions.

%

\subsection{Logarithmic regularity and associated Kato fans}\label{Sec:associated-Kato}

In this section we will work only with logarithmic structures admitting charts \emph{Zariski-locally}.

Let $X$ be a logarithmic scheme and $x$ a schematic point of $X$.  Let $I(x,M_X)$ be the ideal of the local ring $\mathcal{O}_{X,x}$ generated by the maximal prime ideal of $M_{X,x}$.

\begin{definition}[{\cite[Definition~(2.1)]{Kato-toric}}]
A locally noetherian logarithmic scheme $X$ admitting Zariski-local charts is called \emph{logarithmically regular} at a schematic point $x$ if it satisfies the following two conditions:
\begin{enumerate}[(i)]
\item the local ring $\mathcal{O}_{X,x} / I(x,M_X)$ is regular, and
\item $\dim \mathcal{O}_{X,x} = \dim \mathcal{O}_{X,x} / I(x,M) + \rank \onM^{\rm gp}_{X,x}$.
\end{enumerate}
\end{definition}

\begin{example}[{\cite[Example~(2.2)]{Kato-toric}}]
A toric variety with its toric logarithmic structure is logarithmically regular.
\end{example}

If $X$ is a logarithmic scheme then the Zariski topological space of $X$ is equipped with a sheaf of sharp monoids $\onM_X$.  Thus $(X, \onM_X)$ is a sharply monoidal space.  Moreover, morphisms of logarithmic schemes induce morphisms of sharply monoidal spaces.  We therefore obtain a functor from the category of logarithmic schemes to the category of sharply monoidal spaces.  We may speak in particular about morphisms from schemes to Kato fans.

\begin{theorem}[cf.\ {\cite[(10.2)]{Kato-toric}}] \label{thm:kato-univ-prop}
Let $X$ be a fine, saturated locally noetherian, logarithmically regular logarithmic scheme that admits charts \emph{Zariski locally}.  Then there is an initial strict morphism  $X\to F_X$ to a Kato fan.
\end{theorem}

\noindent
We call the Kato fan $F_X$ {\em the Kato fan associated to $X$}.

Kato constructs the Kato fan $F_X$ as a skeleton of the logarithmic strata of $X$.  Let $\undernorm{F}_X \subset X$ be the set of points $x \in X$ such that $I(x,M_X)$ coincides with the maximal ideal of $\mathcal{O}_{X,x}$.  As we indicate below these are the generic points of the logarithmic strata of $X$.  Let $M_{F_X}$ be the restriction of $\onM_X$ to $\undernorm{F}_X$. We denote the resulting monoidal space by $F_X$.

\begin{lemma}[{\cite[Proposition~(10.1)]{Kato-toric}}]
If $X$ is a fine, saturated, locally noetherian, logarithmically regular scheme admitting a chart Zariski locally then the sharply monoidal space constructed above is a fine, saturated Kato fan.
\end{lemma}
\begin{lemma}[{\cite[(10.2)]{Kato-toric}}]
There is a canonical continuous retraction of $X$ onto its Kato fan $F_X$.
\end{lemma}

We briefly summarize the definition of the map and omit the rest of the proof. Let $x$ be a point of $X$.  The quotient $\cO_{X,x} / I(x,M_X)$ is regular, hence in particular a domain, so it has a unique minimal prime corresponding to a point $y$ of $F_X\subset X$.  We define $\pi(x) = y$.

\begin{lemma}
The Kato fan of $X$ is the initial Kato fan admitting a morphism from $X$.
\end{lemma}
\begin{proof}
Let $\varphi : X \rightarrow F$ be another morphism to a Kato fan.  By composition with the inclusion $F_X \subset X$ we get a map $F_X \rightarrow F$.  We need to show that $\varphi$ factors through $\pi : X \rightarrow F_X$. 

Let $x$ be a point of $X$ and let $A$ be the local ring of $x$ in $X$.  Let $P = \onM_{X,x}$.  Let $y = \pi(x)$.  Then $y$ is the generic point of the vanishing locus of $I(x,M_X)$.  We would like to show $\varphi(y) = \varphi(x)$.  We can replace $X$ with $\Spec A$, replace $F_X$ with $F_X \cap \Spec A$, and replace $F$ with an open Kato cone in $F$ containing $\varphi(x)$.  Then $F = \Spec Q$ for some fine, saturated, sharp monoid $Q$ and we get a map $Q \rightarrow P$.  Since $X \rightarrow F$ is a morphism of \emph{sharply} monoidal spaces, $X \rightarrow F$ factors through the open subset defined by the kernel of $Q \rightarrow P$, so we can replace $F$ by this open subset and assume that $Q \rightarrow P$ is sharp.  But then $X \rightarrow F$ factors through no smaller open subset, so $\varphi(x)$ is the closed point of $F$.  Moreover, we have $\onM_{X,y} = \onM_{X,x}$ so the same reasoning applies to $y$ and shows that $\varphi(y) = \varphi(x)$, as desired.
\end{proof}

This completes the proof of Theorem~\ref{thm:kato-univ-prop}.

\subsection{Towards the monoidal analogues of algebraic spaces}

\subsubsection{The need for a more general approach: the nodal cubic}
\label{sec:cone-space-motivation}

Not every logarithmically smooth scheme has a Kato fan.  For example, the divisorial logarithmic structure on $\AA^2$ associated to a nodal cubic curve does not have a chart in any Zariski neighborhood of the node of the cubic.  The Kato fan of this logarithmic structure wants to be the Kato fan of the plane with the open subsets corresponding to the complements of the axes glued together.

\begin{figure}[htb]
\begin{center}

\begin{tikzpicture}

\path [pattern=north west lines, pattern color=black!20!white] (-3,1) -- (0,0) .. controls +(160:0.1cm) and +(-90:0.1cm) .. (-0.2,1) -- (-3,1);
\path [pattern=north east lines, pattern color=black!20!white] (-3,1) -- (-0.2,1) .. controls +(90:0.1cm) and +(200:0.1cm) .. (0,2) -- (-3,1);

\draw (-3,1) -- (0,0);

\fill (-3,1) circle (0.030cm);


\draw [dotted, pattern=dots, pattern color=black!10!white] (0,0) .. controls  +(-90:0.0cm) .. (0.2,1) .. controls +(90:0.1cm) and +(-20:0.1cm) .. (0,2);
\draw [dashed, pattern=dots, pattern color=black!10!white](0,0) .. controls  +(-90:0.0cm) .. (-0.2,1) .. controls +(90:0.1cm) and +(200:0.1cm) .. (0,2);

\end{tikzpicture}
\caption{The fan of the nodal cubic drawn as a cone with the two edges glued to each other.}\label{Fig:Selfintersection}
\end{center}
\end{figure}
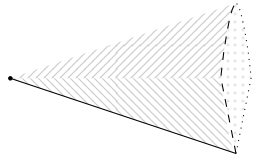

This cannot be a Kato fan, because there is no open neighborhood of the closed point that is a Kato cone.  Indeed, the \ChDan{closed} point has  two different generizations to the codimension~$1$ point. 

Clearly, the solution here is to allow more general types of gluing (colimits) into the definition of a Kato fan. The purpose of this section is to outline what this might entail. Our favorite solution will only come in the next section, where we discuss Artin fans.

 The theory of algebraic spaces provides a blueprint for how to proceed.  The universal way to add colimits to a category is to pass to its category of presheaves.  In order to retain the Kato fans as colimits of their open Kato subcones, we look instead at the category of sheaves.  Finally, we restrict attention to those sheaves that resemble Kato cones \emph{\'etale locally}.

\subsubsection{The need for a more general approach: the Whitney umbrella}\label{sec:cone-stack-motivation}
These \emph{cone spaces} are general enough to include a fan for the divisorial logarithmic structure of the nodal cubic.  However, one cannot obtain a fan for the punctured Whitney umbrella this way:

Working over $\CC$,  let $X = \AA^2 \times \Gm$ with the logarithmic structure pulled back from the standard logarithmic structure on $\AA^2$\ChDan{, associated to the divisor $xy=0$ as in Example \ref{Ex:divisorial-log-structure}}.  Let $G = \ZZ / 2 \ZZ$ \ChDan{act by} $t . (x,y,z) = (y,x,-z)$.  \ChDan{Since the divisor is $G$ stable,} the action lifts to the logarithmic structure, so the logarithmic structure descends to the quotient $Y = X / G$.

There is a projection $\pi : Y \rightarrow \Gm$, and traversing the nontrivial loop in $\Gm$ acts by the automorphism of $\AA^2$ exchanging the axes.  Thus the logarithmic structure of $Y$ has \emph{monodromy}.

If there is a map from $Y$ to a cone space $Z$, then it is not possible for $\onM_Y$ to be pulled back from $M_Z$.  Just as in a Kato fan, the strata of $Z$ on which $M_Z$ is locally constant are discrete and therefore cannot have monodromy.

However, \emph{Deligne--Mumford stacks} can have monodromy at points.  By enlarging our perspective to include \emph{cone stacks}, the analogues of Deligne--Mumford stacks for Kato fans, we are able to construct fans that record the combinatorics of the logarithmic strata for any locally connected logarithmic scheme (logarithmic smoothness is not required).  The construction proceeds circuitously, by showing cone stacks are equivalent to Artin fans (defined in Section~\ref{sec:artin}) and then constructing an Artin fan associated to any logarithmically smooth scheme.

\begin{figure}[htb]
\begin{center}

\begin{tikzpicture}

\draw [white, pattern=north west lines, pattern color=black!20!white] (0,0) -- (2,0) -- (2,2) -- (0,0);

\fill (0,0) circle (0.050cm);

\draw [dashed] (0,0) -- (2,2);
\draw (0,0) -- (2,0);


\end{tikzpicture}

\caption{The fan of the Whitney umbrella drawn as a cone\ChDan{---the first quadrant---folded over} itself via the involution \ChDan{$(x,y)\mapsto (y,x)$}.}\label{Fig:Monodromy}
\end{center}
\end{figure}
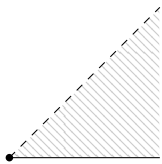

\section{Artin fans}
\label{sec:artin}
In order to simplify the discussion we work over an algebraically closed field $\mathbf{k}$.

\subsection{Definition and basic properties}

\begin{definition} \label{def:artin-fan}
An \emph{Artin fan} is a logarithmic algebraic stack that is logarithmically \'etale over $\Spec \mathbf{k}$.  A morphism of Artin fans is a morphism of logarithmic algebraic stacks.  We denote the $2$-category of Artin fans by the symbol $\AF$.
\end{definition}

Olsson showed that there is an algebraic stack $\underline \Log$ over the category of schemes such that morphisms $\uS \rightarrow \underline\Log$ correspond to logarithmic structures $S$ on $\uS$ \cite[Theorem~1.1]{Olsson_log}.%
\footnote{Our conventions differ from Olsson's:  In \cite{Olsson_log}, the stack $\mathcal{L}\!\mathit{og}$ parameterizes all fine logarithmic structures; to conform with our convention that all logarithmic structures are fine and saturated, we use the symbol $\Log$ to refer to the open substack of Olsson's stack parameterizing fine and saturated logarithmic structures.  This was denoted $\mathcal{T}\!\!\mathit{or}$ in \cite{Olsson_log}.}
Equipping $\underline{\Log}$ with its universal logarithmic structure yields the logarithmic algebraic stack $\Log$.  If $\mathcal{X}$ is an Artin fan then the morphism $\undercal{X} \rightarrow \underline\Log$ induced by the logarithmic structure of $\mathcal{X}$ is \'etale, and conversely:  any logarithmic algebraic stack whose structural morphism to $\underline \Log$ is \'etale is an Artin fan.

\begin{example} \label{ex:artin-cones}
Let $V$ be a toric variety with dense torus $T$.  The toric logarithmic structure of $V$ is $T$-equivariant, hence descends to a logarithmic structure on $[V/T]$, making $[V/T]$ into an Artin fan.

If $V = \Spec \mathbf{k}[M]$ for some fine, saturated, sharp monoid $M$ then $[V/T]$ represents the following functor on logarithmic schemes~\cite[Proposition~5.17]{Olsson_log}:
\begin{equation*}
(X, M_X) \mapsto \Hom\bigl(M, \Gamma(X, \onM_X)\bigr) .
\end{equation*}
As $\onM_X$ is an \'etale sheaf over $X$, it is immediate that $[V/T]$ is logarithmically \'etale over a point.
\end{example}

\begin{definition} \label{def:artin-cones}
An \emph{Artin cone} is an Artin fan isomorphic to $\mathcal{A}_M = [V/T]$, where $V = \Spec \mathbf{k}[M]$ and $M$ is a fine, saturated, sharp monoid.
\end{definition}

\begin{figure}[htb]
\begin{center}


\includegraphics[scale=0.6]{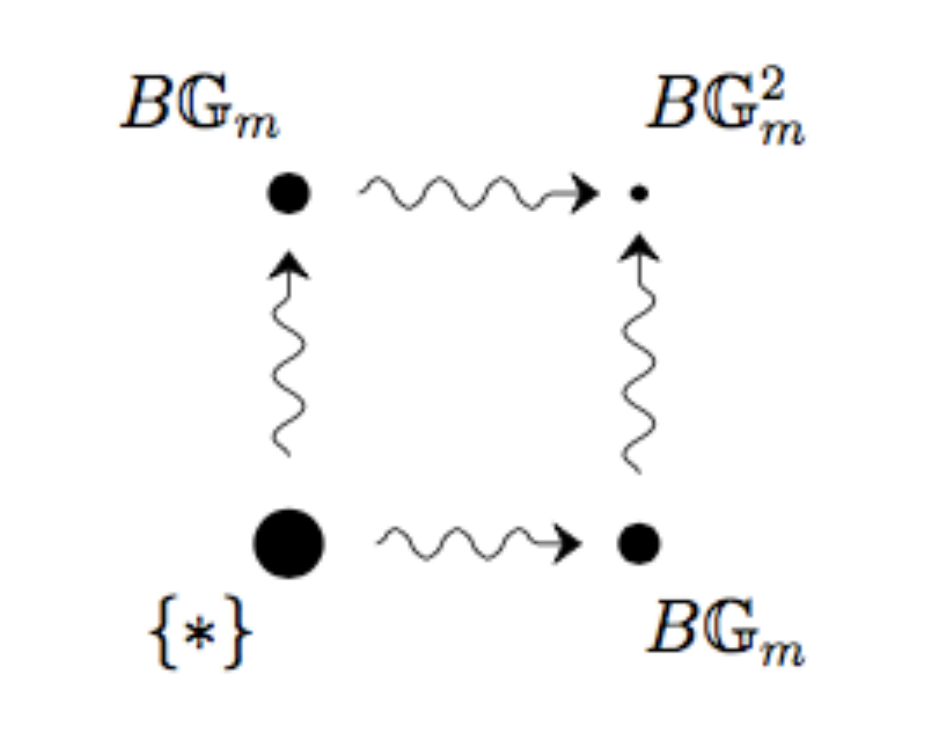}
\caption{$\cA_{\NN^2}.$ The small closed point is $B\Gm^2$, the intermediate points are  $B\Gm$ and the big point is just a point.}\label{Fig:A2}
\end{center}
\end{figure}

If $\mathcal{X} \rightarrow \Log$ is \'etale then $\mathcal{X}$ is determined by its \'etale stack of sections over $\Log$.  Moreover, any \'etale stack on $\Log$ corresponds to an Artin fan by passage to the espace (champ) \'etal\'e.  The following lemma characterizes the \'etale site of $\Log$ as a category of presheaves:

\begin{lemma} \label{lem:etale-site-of-Log}
\begin{enumerate}[(i)]
\item The Artin cones are an \'etale cover of $\Log$.
\item \ChJonathan{An Artin cone has no nontrivial \'etale covers.} 
\item If $M$ and $N$ are fine, saturated, sharp monoids then 
\begin{equation*}
\Hom_{\AF}(\mathcal{A}_M, \mathcal{A}_N) = \Hom_{\mathbf{Mon}}(N,M).
\end{equation*}
\end{enumerate}
\end{lemma}
\begin{proof}
Statement (i) was proved in \cite[Corollary~5.25]{Olsson_log}.  It follows from the fact that every logarithmic scheme admits a chart \'etale locally.

Statement (ii)  follows from  \cite[Corollary 2.4.3]{AW}.  Concretely,   \'etale covers of Artin cones correspond to equivariant \'etale covers of toric varieties, which restrict to equivariant \'etale covers of tori, of which there are none other than the trivial ones \cite[Proposition~2.4.1]{AW}.

Statement (iii) follows from $\Gamma(\mathcal{A}_M, \onM_{\mathcal{A}_M}) = M$ and consideration of the functor represented by $\mathcal{A}_N$ (Example~\ref{ex:artin-cones}).
\end{proof}

\subsection{Categorical context} Lemma~\ref{lem:etale-site-of-Log}~(iii) enables us to relate the 2-category of Artin fans to the notions surrounding  Kato fans.  Let us write $\RPC$ for the category of rational polyhedral cones.  The category $\RPC$ is equivalent to the opposite of the category of fine, saturated, sharp monoids.%
\footnote{In fact, the category of fine, saturated, sharp monoids is equivalent to its own opposite.  However, we find it helpful to maintain a distinction between cones (which we view as spaces) and monoids (which we view as functions).}
Therefore $\RPC$ is equivalent to the category of Kato cones and, by Lemma~\ref{lem:etale-site-of-Log}~(iii), to the category of Artin cones.  Furthermore, we obtain

\begin{corollary}
The $2$-category $\AF$ of Artin fans is fully faithfully embedded in the $2$-category of fibered categories over $\RPC$.
\end{corollary}
\begin{proof}
Lemma~\ref{lem:etale-site-of-Log} gives an embedding of $\AF$  in the the $2$-category of fibered categories 
on the category of Artin cones and identifies the category of Artin cones with $\mathbf{RPC}$.
\end{proof}

This enables us to relate the 2-category $\AF$ with the framework proposed in Sections \ref{sec:cone-space-motivation} and
\ref{sec:cone-stack-motivation}:  the 2-category of cone stacks suggested in Section \ref{sec:cone-stack-motivation} is necessarily equivalent to the 2-category $\AF$. The key to proving this is the fact that the diagonal of an Artin fan is represented by algebraic spaces, which enables one to relate it to a morphism of ``cone spaces" as suggested in Section \ref{sec:cone-space-motivation}. 

In particular, we have a fully faithful embedding $\KF \to \AF$ of the category of Kato fans in the 2-category of Artin fans.

\subsection{The Artin fan of a logarithmic scheme}\label{sec:artinfanofascheme}

\begin{definition}\label{definition_smalllogsch}
A Zariski logarithmic scheme $X$ is said to be {\em small}
with respect to a point $x\in X$,
if the restriction morphism $\Gamma(X,\Mbar)\rightarrow \Mbar_{X,x}$ is an isomorphism and the closed logarithmic stratum
\begin{equation*}
\big\{y\in X\big\vert\Mbar_{X,y}\simeq\Mbar_{X,x}\big\}
\end{equation*}
is {\em connected}. We say $X$ is {\em small} if it is small with respect to some point.
\end{definition}

Let $N = \Gamma(X, \onM_X)$.  There is a canonical morphism
\begin{equation*}
X \rightarrow \mathcal{A}_N
\end{equation*}
corresponding to the morphism $N \rightarrow \Gamma(X, \onM_X)$.  It is shown in \cite{ACMW} that {\em if $X$ is small} this morphism is initial among all morphisms from $X$ to Artin fans.  We therefore call $\mathcal{A}_N$ the \emph{Artin fan} of such small $X$.  By the construction, the Artin fan of $X$ is functorial with respect to strict morphisms.

Now consider a logarithmic algebraic stack $X$ with a groupoid presentation
\begin{equation*}
V \rightrightarrows U \rightarrow X 
\end{equation*}
in which $U$ and $V$ are disjoint unions of {\em small} Zariski logarithmic schemes.  Then $V$ and $U$ have Artin fans $\mathcal{V}$ and $\mathcal{U}$ and we obtain strict morphisms of Artin fans $\mathcal{V} \rightarrow \mathcal{U}$.  Strict morphisms of Artin fans are \'etale, so this is a diagram of \'etale spaces over $\Log$.  It therefore has a colimit, also an \'etale space over $\Log$, which we call the Artin fan of $X$.

\subsection{Functoriality of Artin fans: problem and fix}

The universal property of the Artin fan implies immediately that Artin fans are functorial with respect to \emph{strict} morphisms of logarithmic schemes.  They are not functorial in general, but we will be able to salvage a weak replacement for functoriality in which morphisms of logarithmic schemes induce correspondences of Artin fans.

\subsubsection{The failure of functoriality}

We use the notation for the punctured Whitney umbrella introduced in Section~\ref{sec:cone-stack-motivation}.  As $X$ has a global chart, its Artin fan $\mathcal{X}$ is easily seen to be $\mathcal{A}^2$.  The Artin fan $\mathcal{Y}$ of $Y$ is the quotient of $\mathcal{A}^2$ by the action of $\ZZ / 2 \ZZ$ exchanging the components, \emph{as a representable \'etale space over $\Log$}.  In other words, the group action induces an \'etale equivalence relative to $\Log$ by taking the image of the action map
\begin{equation*}
\ZZ / 2 \ZZ \times \mathcal{A}^2 \rightarrow \mathcal{A}^2 \mathop{\times}_{\Log} \mathcal{A}^2 .
\end{equation*}
A logarithmic morphism from a logarithmic scheme $S$ into $\mathcal{A}^2 \mathop{\times}_{\Log} \mathcal{A}^2$ consists of two maps $\mathbf{N}^2 \rightarrow \Gamma(S, \onM_S)$ and an isomorphism between the induced logarithmic structures $M_1$ and $M_2$ commuting with the projection to $M_S$.  This implies that
\begin{equation*}
\mathcal{A}^2 \mathop{\times}_{\Log} \mathcal{A}^2 = \mathcal{A}^2 \mathop{\amalg}_{\mathcal{A}^0} \mathcal{A}^2
\end{equation*}
where by $\mathcal{A}^0$ we mean the open point of $\mathcal{A}^2$.  The nontrivial projection \ChJonathan{in}
\begin{equation*}
\mathcal{A}^2 \mathop{\amalg}_{\mathcal{A}^0} \mathcal{A}^2 = \mathcal{A}^2 \mathop{\times}_{\Log} \mathcal{A}^2 \rightrightarrows \mathcal{A}^2
\end{equation*}
is given by the identity map on one component and the exchange of coordinates on the other component.  (The trivial projection is the identity on both components.)

It is now easy to see that \ChJonathan{$\ZZ / 2 \ZZ \times \mathcal{A}^2$} surjects onto $\mathcal{A}^2 \mathop{\amalg}_{\mathcal{A}^0} \mathcal{A}^2$, so the Artin fan of $Y$ is the image of $\mathcal{A}^2$ in $\Log$.  We will write $\mathcal{A}^{[2]}$ for this open substack and observe that it represents the functor sending a logarithmic scheme $S$ to the category of pairs $(\onM, \varphi)$ where $\onM$ is an \'etale sheaf of monoids on $S$ and $\varphi : \onM \rightarrow \onM_S$ is a strict morphism \emph{that can be presented \'etale locally by a map $\mathbf{N}^2 \rightarrow \onM_S$}.  Equivalently, it is a logarithmic structure over $M_S$ that has \'etale-local charts by $\mathbf{N}^2$.

Of course, there is a map $\mathcal{X} \rightarrow \mathcal{Y}$ consistent with the maps $X \rightarrow Y$:  it is the canonical projection $\mathcal{A}^2 \rightarrow \mathcal{A}^{[2]}$.  However, we take $\widetilde{X}$ to be the logarithmic blowup of $X$ at $\{ 0 \} \times \Gm$ and let $\widetilde{Y}$ be the logarithmic blowup of $Y$ at the image of this locus.  Then we have a cartesian diagram, since $X$ is flat over $Y$:
\begin{equation*} \xymatrix{
\widetilde{X} \ar[r] \ar[d] & \widetilde{Y} \ar[d] \\
X \ar[r] & Y
} \end{equation*}
We will compute the Artin fans of $\widetilde{X}$ and $\widetilde{Y}$.

Since $X$ and $Y$ are flat over their Artin fans, the blowups $\widetilde{X}$ and $\widetilde{Y}$, are pulled back from the blowups $\widetilde{\mathcal{X}}$ and $\widetilde{\mathcal{Y}}$ of $\mathcal{X}$ and $\mathcal{Y}$.  Furthermore, the square in the diagram below is cartesian:
\begin{equation*} \xymatrix{
\widetilde{\mathcal{X}} \ar[r] \ar[d] & \widetilde{\mathcal{Y}} \ar[r] \ar[d] & \mathcal{Z} \ar@{-->}[dl] \\
\mathcal{X} \ar[r] & \mathcal{Y}
} \end{equation*}
We have written $\mathcal{Z}$ for the Artin fan of $\tY$.  Since $\tY \rightarrow \widetilde{\mathcal{Y}}$ is strict and smooth with connected fibers, the map $\tY \rightarrow \mathcal{Z}$ factors through $\widetilde{\mathcal{Y}}$.  We will see in a moment that there is no dashed arrow making the triangle on the right above commutative, thus witnessing the failure of the functoriality of the Artin fan construction with respect to the morphism $\tY \rightarrow Y$.%
\footnote{More specifically, there is no way to make the Artin fan functorial while also making the maps $Y \rightarrow \mathcal{Y}$ natural in $Y$.}

The reason no map $\mathcal{Z} \rightarrow \mathcal{Y}$ can exist making the diagram above commutative is that $\widetilde{\mathcal{Y}}$ has monodromy at the generic point of its exceptional divisor, pulled back from the monodromy at the closed point of $\mathcal{Y}$.  However, the image of the exceptional divisor of $\widetilde{\mathcal{Y}}$ in $\mathcal{Z}$ is a divisor, and $\mathcal{Z} \rightarrow \Log$ is representable by algebraic spaces.  No rank~$1$ logarithmic structure can have monodromy, so there is no monodromy at the image of the exceptional divisor in $\mathcal{Z}$.

A variant of the last paragraph shows that there is no commutative diagram
\begin{equation*} \xymatrix{
\tY \ar[r] \ar[d] & \mathcal{Z} \ar[d] \\
Y \ar[r] & \mathcal{Y} .
} \end{equation*}
We can find a loop $\gamma$ in the exceptional divisor \ChDan{$E$} of $\tY$ that projects to a nontrivial loop in $Y$, around which the logarithmic structure of $Y$ has nontrivial monodromy.  Even though the logarithmic structure of $\tY$ has no monodromy around $\gamma$, the image of $\gamma$ in $\mathcal{Y}$ is nontrivial.  But all of \ChDan{the exceptional divisor $E$} is collapsed to a point in $\mathcal{Z}$, so the image of $\gamma$ in $\mathcal{Y}$ must act trivially.%
\footnote{Here is a rigorous version of the above argument.  Observe that $\mathcal{Y}$ is the quotient of $\mathcal{R} \rightrightarrows \mathcal{X}$, where $\mathcal{X} \simeq \mathcal{A}^2$ and $\mathcal{R}$ is two copies of $\mathcal{X}$ joined along their open point.  Pulling $\mathcal{X}$ back to the center of the blowup $Y_0 \simeq \Gm$ inside $Y$ yields the cover by $\Gm \times \{ 0 \} \subset X$.  This cover has nontrivial monodromy.  On the other hand, $\mathcal{Z} \simeq \mathcal{A}^2$ has no nontrivial \'etale covers.  Therefore the pullback of $\mathcal{X} \mathop{\times}_{\tsY} \mathcal{Z}$ via $\tY \rightarrow \mathcal{Z}$ yields a trivial \'etale cover of $Y_0$.}

\subsubsection{The patch}
\label{sec:patch}

There seem to be two ways to get around this failure of functoriality.  The first is to allow the Artin fan to include more information about the fundamental group of the original logarithmic scheme $X$.  However, the most naive application of this principle would introduce the entire \'etale homotopy type of $X$ into the Artin fan, sacrificing Artin fans' essentially combinatorial nature.

Another approach draws inspiration from Olsson's stacks of diagrams of logarithmic structures~\cite{Olsson-log-cc}.  Let $\Log^{[1]}$ be the stack whose $S$-points are morphisms of logarithmic structures $M_1 \rightarrow M_2$ on $S$.  A morphism of logarithmic schemes $X \rightarrow Y$ induces a commutative diagram
\begin{equation*} \xymatrix{
X \ar[r] \ar[d] & \Log^{[1]} \ar[d] \\
Y \ar[r] & \Log
} \end{equation*}
where $\Log^{[1]} \rightarrow \Log$ sends $M_1 \rightarrow M_2$ to $M_1$.  If $\Log^{[1]}$ is given $M_2$ as its logarithmic structure, this is a commutative diagram of logarithmic algebraic stacks.  The construction of the Artin fan works in a relative situation, and we take $\mathcal{Y} = \pi_0(Y / \Log)$ and $\mathcal{X} = \pi_0(X / \Log^{[1]})$.  Note that $\mathcal{Y} \mathop{\times}_{\Log} \Log^{[1]}$ is \'etale over $\Log^{[1]}$, so we get a map 
\begin{equation*}
\mathcal{X} \rightarrow \mathcal{Y} \mathop{\times}_{\Log} \Log^{[1]} \rightarrow \mathcal{Y}
\end{equation*}
salvaging a commutative diagram:

\begin{theorem}[\protect{\cite[Corollary 3.3.5]{AW}}]\label{thm:artinfunctoriality}
For any morphism of logarithmic schemes $X \rightarrow Y$ with \'etale-locally connected logarithmic strata there is an initial commutative diagram
\begin{equation*} \xymatrix{
X \ar[r] \ar[d] & \mathcal{X} \ar[d] \ar[r] & \Log^{[1]} \ar[d] \\
Y \ar[r] & \mathcal{Y} \ar[r] & \Log
} \end{equation*}
in which the horizontal arrows are strict and both $\mathcal{X}$ and $\mathcal{Y}$ are Artin fans representable by algebraic spaces relative to $\Log^{[1]}$ and $\Log$, respectively.
\end{theorem}

\section{Algebraic applications of Artin fans}
\label{sec:obs}

\subsection{Gromov--Witten theory and relative Gromov--Witten theory} Algebraic Gromov--Witten theory is the study of the virtually enumerative invariants, known as Gromov--Witten invariants, of algebraic curves on a smooth target variety $X$. In Gromov--Witten theory one integrates cohomology classes on $X$ against the virtual fundamental class $[\ocM_\Gamma(X)]^{\rm vir}$ of the moduli space $\ocM_\Gamma(X)$ of stable maps with target $X$.  The subscript $\Gamma$ indicates fixed numerical invariants, including the genus of the domain curve, the number of marked points on it, and the homology class of its image.

{\em Relative Gromov--Witten theory} comes from efforts to define Gromov--Witten invariants for degenerations of complicated targets that, while singular, are still geometrically simple.  In the mild setting of two smooth varieties meeting along a smooth divisor, such a theory has been developed by J.\ Li \cite{Li1,Li2}, following work in symplectic geometry by A.\ M.\ Li--Ruan \cite{Li-Ruan} and Ionel--Parker \cite{Ionel-Parker, Ionel-Parker2}.  

Working over $\CC$, one considers a degeneration  
\begin{equation}\label{fiberofacceptabledegeneration}
\vcenter{\xymatrix{   Y_1 \mathbin{\sqcup}_D Y_2\ar[d]\ar@{}[r]|-*[@]{\subset} & X\ar[d]^\pi\\
b_0\ar@{}[r]|-*[@]{\subset}  & B
}}
\end{equation}
where $\pi:X\to B$ is a flat, projective morphism from a smooth variety to a smooth curve, and where  the singular fiber $X_0=Y_1\sqcup_D Y_2$ consists of two smooth varieties meeting along a smooth divisor.  

Jun Li proved an algebro-geometric degeneration formula through which one can recover Gromov--Witten invariants of the possibly complicated but smooth general fiber of  $X$ from \emph{relative Gromov--Witten invariants} determined by a space $\ocM_{\Gamma_i}(Y_i)$ of {\em relative stable maps} to each of the two smooth components $Y_i$ of $X_0$.  Here, in addition to the genus, number of markings, and homology class, one must fix the {\em contact orders} of the curve with the given divisor.

\subsubsection{Expanded degenerations and pairs} In order to define Gromov--Witten invariants of the {\em singular} degenerate fiber $X_0$, Jun Li constructed a whole family of {\em expansions} $X_0' \to X_0$, where
\[X_0' = Y_1\sqcup_D P \sqcup_D \cdots \sqcup_D P \sqcup_D Y_2.\] 
Here  $P$ is the projective completion of $N_{D/Y_1} \simeq N_{D/Y_2}^\vee$ (explicitly, it is $\PP(\cO\oplus N_{D/Y_1})$) and the gluing over $D$ attaches $0$-sections to $\infty$-sections.

Similarly,  in order to guarantee that contact orders of maps in each $Y_i$ are maintained,  Jun Li constructed a family of expansions $Y_i' \to Y_i$ where \[Y_i' = Y_i\sqcup_D P \sqcup_D \cdots \sqcup_D P.\]

Here is the first point where Artin fans, in their simplest form and even without logarithmic structures, become of use: 

In Jun Li's construction, not every deformation of an expansion $Y_i'\to Y_i$ is itself an expansion. For instance, the expansion $Y_i\sqcup_D P$ can deform to  $Y_i\sqcup_D P'$, where $P''= \PP(\cO\oplus N'')$ with $N''$ a deformation of $N_{D/Y_1}$.  Precisely the same problem occurs with deformation of an expansion $X_0' \to X_0$.

\subsubsection{The Artin fan as the universal target, and its expansions}
In \cite{ACFW}, following ideas in \cite{cadman}, it was noted that the $Y_i$ has a canonical map  $Y_i \to\cA := [\AA^1/\Gm]$. From the point of view of the present text,  $\cA = \cA_{Y_i}$, the Artin fan associated to the divisorial logarithmic structure $(Y_i,D)$, and its divisor $\cD = B\G_m\subset \cA$ is the universal divisor. Next, if $\cA'\to \cA$ is an expansion, then {\em all deformations of $\cA'\to \cA$ are expansions of $\cA$ in the sense of Jun Li.}  and finally, any expansion $Y_i'\to Y_i$ is obtained as the pullback $Y_i' = \cA'\times_{\cA} Y_i$ of some expansion $\cA'\to \cA$. 

This means that the moduli space of \ChDan{expansions} of any pair $(Y_i,D)$ is identical to the moduli space of expansions of $(\cA,\cD)$, and the \ChDan{expansions} themselves are obtained by pullback.

A similar picture occurs for degenerations: the Artin fan of $X \to B$ is the morphism $\cA \times \cA \to \cA$ induced by the multiplication morphism $\AA^2 \to \AA^1$:
\begin{equation*}\label{canonicalsquare}
\xymatrix{ X\ar[d]\ar[r] & \cA^2\ar[d]\\
B\ar[r]  & \cA.
}
\end{equation*}
 There is again a stack of universal expansions  of $\cA \times \cA \to \cA$, and every expansion $X_0' \to X_0$ is the pullback of a fiber of the universal expansion:
 \[
\xymatrix{ X_0'\ar[r]\ar[d] & (\cA^2)' \ar[d]\\
X_0\ar[r] & \cA^2\times_A B}
\]

From the point of view of logarithmic geometry this is not surprising: expansions are always stable under {\em logarithmic} deformations. But the approach through Artin fans provides us with further results, which we outline below.

\subsubsection{Redefining obstructions} Using expanded degenerations and expanded pairs, Jun Li defined moduli spaces of {\em degenerate stable maps} $\ocM_\Gamma(X/B)$ and of {\em relative stable maps} $\ocM_{\Gamma_i}(Y_i, D)$.  Jun Li had an additional challenge in defining the virtual fundamental classes of these spaces, which he constructed by bare hands. With a little bit of hindsight, we now know that Li's virtual fundamental classes are associated to the natural {\em relative obstruction theories} of the morphisms  $\ocM_\Gamma(X/B) \to  \mathfrak M_\Gamma(\cA^2/\cA)$ and  $\ocM_{\Gamma_i}(Y_i, D)\to \mathfrak M_{\Gamma_i}(\cA,\cD)$, where $\mathfrak M_\Gamma(\cA^2/\cA)$ and $\mathfrak M_{\Gamma_i}(\cA,\cD)$ are the associated moduli spaces of {\em prestable maps}. Moreover, the virtual fundamental classes can be understood with machinery available off the shelf of any deformation theory emporium.  This observation from \cite{AMW} made it possible to prove a number of comparison results, including those described below.

\subsubsection{Other approaches and comparison theorems}

 Denote by $(Y,D)$ a smooth pair, consisting of a smooth projective variety $Y$ with a smooth and irreducible divisor $D$.    Jun Li's moduli space $\ocM_\Gamma(Y,D)$ of relative stable maps to $(Y,D)$ provides an algebraic setting for relative Gromov--Witten theory.  Only recently have efforts to generalize the theory to more complicated singular targets come to fruition \cite{Parker,Parker1,GS,AC, ionel}. Because of the technical difficulty of Jun Li's approach,  several alternate approaches to the relative Gromov--Witten invariants of $(Y,D)$ have been developed:

\begin{itemize} 
\item[$\bullet$] $\JListab_\Gamma(Y,D):$ Li's original moduli space of relative stable maps;
\item[$\bullet$] $\AFstab_\Gamma(Y,D):$ Abramovich--Fantechi's stable orbifold maps with expansions~\cite{AF};
\item[$\bullet$] $\Kimstab_\Gamma(Y,D):$ Kim's  stable logarithmic maps with expansions~\cite{Kim};
\item[$\bullet$] $\ACGSstab_\Gamma(Y,D):$ Abramovich--Chen and Gross--Siebert's stable logarithmic maps without expansions~\cite{Chen,GS,AC}.
\end{itemize} 

There are analogous constructions 
\begin{equation*}
\JListab_\Gamma(X/B)\qquad \AFstab_\Gamma(X/B)\qquad \Kimstab_\Gamma(X/B)\qquad \ACGSstab_\Gamma(X/B)
\end{equation*}
for a degeneration. 

This poses a new conundrum: how do these approaches compare? The answer, which depends on the Artin fan $\cA$, is as follows:

\begin{theorem} \label{thm:compare}\cite[Theorem~1.1]{AMW}
There are maps
\[
 \xymatrix{
     \cat{AF}  \ar[dr]^{\Psi} & & \Kim \ar[dl]_{\Theta} \ar[dr]^{\Upsilon} & \\
  & \JLi &  & \ACGS .
}
\]
such that
\begin{align*}
\Psi_\ast [\cat{AF}]^{\vir} &= [\JLi]^{\vir} &
\Theta_\ast [\Kim]^\vir &= [\JLi]^\vir &
\Upsilon_\ast [\Kim]^\vir &= [\ACGS]^\vir.
\end{align*}
In particular, the Gromov--Witten invariants associated to these four theories coincide.
\end{theorem}

The principle behind this comparison for each of the three maps $\Psi, \Theta$ and $\Upsilon$ is the same, and we illustrate it on $\Psi$: there are algebraic stacks \ChDan{parametrizing orbifold and relative stable maps to the Artin fan $(\cA,\cD)$,}  which we denote here by  $\cat{AF}_\Gamma(\cA,\cD)$ and $\JLi_\Gamma(\cA,\cD)$\ChDan{. These stacks sit} in a cartesian diagram 
$$\xymatrix{\AFstab_\Gamma(\ChDan{\ChJonathan{Y},D}) \ar[r]^\Psi\ar[d]_{\pi_{\cat{AF}}}& \JListab_\Gamma(Y,D) \ar[d]^{\pi_{\JLi}} \\
\cat{AF}_\Gamma(\cA,\cD) \ar[r]^{\Psi_\cA} & \JLi_\Gamma(\cA,\cD)
}
$$
such that 
\begin{enumerate}
\item The virtual fundamental classes of the spaces $\AFstab_\Gamma(\ChDan{\ChJonathan{Y},D})$ and $\JListab_\Gamma(Y,D)$ may be computed using the natural obstruction theory relative to the vertical arrows $\pi_{\cat{AF}}$ and $\pi_{\JLi}$.
\item The obstruction theory for $\pi_{\cat{AF}}$ is the pullback of the obstruction theory of $\pi_{\JLi}$.
\item The morphism $\Psi_\cA$ is birational.
\end{enumerate}

One then applies a general comparison result of Costello \cite[Theorem 5.0.1]{costello} to obtain the theorem.

\ChDan{What makes all this possible is the fact that  the virtual fundamental classes of $\cat{AF}_\Gamma(\cA,\cD)$ and $\JLi_\Gamma(\cA,\cD)$  agree with their fundamental classes. This in turn results from the fact that the Artin fan $(\cA,\cD)$ of $(X,D)$ discards all the complicated geometry of $(X,D)$, retaining  just enough algebraic structure to afford stacks of maps such as  $\cat{AF}_\Gamma(\cA,\cD)$ and $\JLi_\Gamma(\cA,\cD)$.}  \ChJonathan{In effect, the \emph{virtual} birationality of $\Psi$ is due to the \emph{genuine} birationality of $\Psi_{\cA}$.}

Similar results were obtained by similar methods in  \cite{CMW} and \cite{ACW}.

\subsection{Birational invariance for logarithmic stable maps}\label{ss:virbirationality}

More general Artin fans have found applications in the logarithmic approach to Gromov--Witten theory. 

Let $X$ be a projective {\em logarithmically smooth} variety over $\CC$, and denote by $\ocM(X)$ the moduli space of {\em logarithmic} stable maps to $X$, as defined in \cite{GS, Chen, AC}.  Fix a logarithmicaly \'etale morphism $h:Y \to X$ of projective logarithmically smooth varieties.   There is a natural morphism $\ocM(h) :\ocM(Y) \to \ocM(X)$ induced by $h$, and the central result of \cite{AW} is the following pushforward statement for virtual classes.

\begin{theorem}\label{AWmaintheorem} $\ocM(h)_* \left( [\ocM(Y)]^{\rm vir}\right) =  [\ocM(X)]^{\rm vir}.$
In particular the logarithmic Gromov--Witten invariants of $X$ and $Y$ coincide.
\end{theorem}

This theorem is proven by working relative to the underlying morphism of Artin fans $\mathcal{Y}\to\mathcal{X}$ provided by Theorem~\ref{thm:artinfunctoriality}.  This morphism and $h$ fit into a cartesian diagram
\begin{equation} \label{diag:lift-map} \vcenter{\xymatrix{
Y \ar[r]^h \ar[d] & X \ar[d] \\
\mathcal{Y} \ar[r] & \mathcal{X}.
}} \end{equation}

One carefully constructs a cartesian diagram of moduli spaces
\begin{equation} \label{eqn:3} \vcenter{\xymatrix{
\ocM(Y) \ar[r]^{\ocM(h)}\ar[d] &  		\ocM(X)\ar[d] \\
\ocM(\mathcal{Y} )\ar[r] &	\ocM(\mathcal{X})}}
\end{equation}
to which principles (1), (2) and (3) in the proof of Theorem  \ref{thm:compare} apply. Costello's comparison theorem \cite[Theorem~5.0.1]{costello} again gives the result.

\subsection{Boundedness of Logarithmic Stable Maps}

A recent application of both Theorem~\ref{AWmaintheorem} and the theory of Artin fans can be found in \cite{ACMW}, where a general statement for the boundedness of logarithmic stable maps to projective logarithmic schemes is proven.

\begin{theorem}[\protect{\cite[Theorem 1.1.1]{ACMW}}]\label{thm:boundedness}
Let $X$ be a projective logarithmic scheme. Then the stack $\ocM_{\Gamma}(X)$ of stable logarithmic maps to $X$ with discrete data $\Gamma$ is of finite type. 
\end{theorem}


The boundedness of $\ocM_{\Gamma}(X)$ has been established when the characteristic monoid $\onM_{X}$ is globally generated  in \cite{AC, Chen}, and more generally when the associated group $\onM_{X}^{\rm gp}$ is globally generated in \cite{GS}. The strategy used in \cite{ACMW} for the general setting is to reduce to the case of a globally generated sheaf of monoids $\onM_{X}$ by studying the behavior of stable logarithmic maps under an appropriate modification of $X$. This is accomplished by modifying the Artin fan $\mathcal{X} = \cA_X$ constructed in Section \ref{sec:artinfanofascheme}, lifting this modification to the level of logarithmic schemes, and applying a virtual birationality result refining Theorem \ref{AWmaintheorem}.  

A key step is the modification of $\mathcal{X}$:

\begin{proposition}\label{prop:resolve-artin-fan}\cite[Proposition 1.3.1]{ACMW}:
Let $\mathcal{X}$ be an Artin fan. Then there exists a projective, birational, and logarithmically \'etale morphism $\mathcal{Y} \to \mathcal{X}$ such that $\mathcal{Y}$ is a smooth Artin fan, and the characteristic sheaf $\onM_{\mathcal{Y}}$ is globally generated and locally free.
\end{proposition}

The proof of this Proposition is analogous to \cite[I.11]{KKMS}: successive star subdivisions are applied until each cone is smooth (see Example~\ref{ex:star}), and barycentric subdivisions guarantee that the resulting logarithmic structure has no monodromy.

\section{Skeletons and tropicalization}
\label{sec:skel}
\subsection{Berkovich spaces}

Ever since \cite{Tate_rigidanalyticspaces} it has been known that affinoid algebras are the correct coordinate rings for defining non-Archimedean analogues of complex analytic spaces. Working with affinoid algebras as coordinate rings alone is enough information to build an intricate theory with many applications. The work of V. Berkovich \cite{Berkovich_book} and \cite{Berkovich_etalecoho} beautifully enriches this theory by providing us with an alternative definition of non-Archimedean analytic spaces, which naturally come with underlying topological spaces that have many of the favorable properties of complex analytic spaces, such as being locally path-connected and locally compact. 

Let $\mathbf{k}$ be a non-Archimedean field, i.e. suppose that $\mathbf{k}$ is complete with respect to a non-Archimedean absolute  value $\vert.\vert$. We explicitly allow $\mathbf{k}$ to carry the trivial absolute value. If $U=\Spec A$ is an affine scheme of finite type over $\mathbf{k}$, as a set the \emph{analytic space} $U^{an}$ associated to $U$ is equal to the set of multiplicative seminorms on $A$ that restrict to the given absolute value on $\mathbf{k}$. We usually write $x$ for a point in $U^{an}$ and $\vert .\vert_x$, if we want to emphasize that $x$ is thought of as a multiplicative seminorm on $A$. The topology on $U^{an}$ is the coarsest that makes the maps
\begin{equation*}\begin{split}
U^{an}&\longrightarrow \R \\
x&\longmapsto \vert f\vert_x
\end{split}\end{equation*}
continuous for all $f\in A$. There is a natural continuous \emph{structure morphism} $\rho\mathrel{\mathop:}U^{an}\rightarrow U$ given by sending $x\in U^{an}$ to the preimage of zero $\big\{f\in A\big\vert \vert f\vert_x=0\big\}$. A morphism $f\mathrel{\mathop:}U\rightarrow V$ between affine schemes $U=\Spec A$ and $V=\Spec B$ of finite type over $\mathbf{k}$, given by a $\mathbf{k}$-algebra homomorphism $f^\#\mathrel{\mathop:}B\rightarrow A$ induces a natural continuous map $f^{an}\mathrel{\mathop:}U^{an}\rightarrow V^{an}$ given by associating to $x\in U^{an}$ the point $f(x)\in V^{an}$ given as the multiplicative seminorm 
\begin{equation*}
\vert.\vert_{f(x)}=\vert.\vert_x\circ f^\#
\end{equation*}
on $B$. The association $f\mapsto f^{an}$ is functorial in $f$.

Let $X$ be a scheme that is locally of finite type over $\mathbf{k}$. Choose a covering $U_i=\Spec A_i$ of $X$ by open affines. Then the \emph{analytic space} $X^{an}$ associated to $X$ is defined by glueing the $U_i^{an}$ over the open subsets $\rho^{-1}(U_i\cap U_j)$. It is easy to see that this construction does not depend on the choice of a covering and that it is functorial with respect to morphisms of schemes over $\mathbf{k}$. We refrain from describing the structure sheaf on $X^{an}$, since we are only going to be interested in the topological properties of $X^{an}$, and refer the reader to \cite{Berkovich_book,Berkovich_etalecoho,Temkin-intro} 
for further details. 

\begin{example}\label{example_A1} Suppose that $\mathbf{k}$ is algebraically closed and endowed with the trivial absolute value. Let $t$ be a coordinate on the the affine line $\AA^1=\Spec \mathbf{k}[t]$. One can classify the points in $(\AA^1)^{an}$ as follows:
\begin{itemize}
\item For every $a\in \mathbf{k}$ and $r\in[0,1)$ the seminorm $\vert.\vert_{a,r}$ on $\mathbf{k}[t]$ that is uniquely determined by
\begin{equation*}
\vert t-a\vert_{a,r}=r \ ,
\end{equation*}
and
\item for $r\in[1,\infty)$ the seminorm $\vert.\vert_{\infty, r}$ uniquely determined by
\begin{equation*}
\vert t-a\vert_{\infty, r}=r
\end{equation*}
for all $a\in \mathbf{k}$. 
\end{itemize}
Noting that 
\begin{equation*}
\lim_{r\rightarrow 1}\vert.\vert_{a,r}=\vert.\vert_{\infty,1}
\end{equation*}
for all $a\in \mathbf{k}$ we can visualize $(\AA^1)^{an}$ as follows:
\begin{center}\begin{tikzpicture}

\draw (0,2) circle (0.08 cm);
\fill (0,0) circle (0.08 cm);
\fill (-2,-2) circle (0.08 cm);
\fill (-1,-2) circle (0.08 cm);
\fill (0,-2) circle (0.08 cm);
\fill (1,-2) circle (0.08 cm);
\fill (2,-2) circle (0.08 cm);

\draw (0,0) -- (0,1.92);
\draw (0,0) -- (-2,-2);
\draw (0,0) -- (-1,-2);
\draw (0,0) -- (0,-2);
\draw (0,0) -- (1,-2);
\draw (0,0) -- (2,-2);

\node at (0,-2.5) {$0$};
\node at (0,2.5) {$\infty$};
\node at (3,-2) {$\A^1(\mathbf{k})$};
\node at (0.5,0) {$\eta$};
\node at (2,1) {$(\A^1)^{an}$};

\fill (-1.7,-1) circle (0.03cm);
\fill (-1.9,-1) circle (0.03cm);
\fill (-2.1,-1) circle (0.03cm);
\fill (1.7,-1) circle (0.03cm);
\fill (1.9,-1) circle (0.03cm);
\fill (2.1,-1) circle (0.03cm);

\end{tikzpicture}\end{center}

In particular, we can embed the closed points \ChJonathan{of} $\AA^1(\mathbf{k})$ into $(\AA^{1})^{an}$ by the association $a\mapsto \vert .\vert_{a,0}$ for \ChDan{$a\in \mathbf{k}$.}
\end{example}

One can give an alternative description of $X^{an}$ as the set of equivalence classes of non-Archimedean points of $X$. A \emph{non-Archimedean point} of $X$ consists of a pair $(K, \phi)$ where $K$ is a non-Archimedean extension of $\mathbf{k}$ and $\phi\mathrel{\mathop:}\Spec K\rightarrow X$. Two non-Archimedean points $(K,\phi)$ and $(L,\psi)$ of $X$ are \emph{equivalent}, if there is a common non-Archimedean extension $\Omega$ of both $K$ and $L$ such that the diagram
\begin{equation*}\begin{CD}
\Spec \Omega @>>>\Spec L\\
@VVV @VV\psi V\\
\Spec K@>>\phi> X
\end{CD}\end{equation*}
commutes.

\begin{proposition}
The analytic space $X^{an}$ is equal to the set of non-Archimedean points modulo equivalence.
\end{proposition}

\begin{proof}
We may assume that $X=\Spec A$ is affine. A non-Archimedean point $(K,\phi)$ of $X$ naturally induces a multiplicative seminorm 
\begin{equation*}
A\xrightarrow{\phi^\#}K\xrightarrow{\vert.\vert}\R
\end{equation*}
on $A$. Conversely, given $x\in X^{an}$, we can consider the integral domain $A/\ker\vert.\vert_x$ and form the completion $\calH(x)$ of its field of fractions. The natural homomorphism $A\rightarrow\calH(x)$ defines a non-Archimedean point on $X$ and these two constructions are inverse up to equivalence. 
\end{proof}

Now suppose that $\mathbf{k}$ is endowed with the trivial absolute value. In \cite[Section 1]{Thuillier} Thuillier introduces a slight variant of the analytification functor that functorially associates to a scheme locally of finite type over $\mathbf{k}$ a non-Archimedean analytic space $X^\beth$. Its points are pairs $(R,\phi)$ consisting of a valuation ring $R$ extending $\mathbf{k}$ together with a morphism $\phi\mathrel{\mathop:}\Spec R\rightarrow X$ modulo an equivalence relation as above: Two pairs $(R,\phi)$ and $(S,\psi)$ are \emph{equivalent}, if there is a common valuation ring $\calO$ extending both $R$ and $S$ such that the diagram
\begin{equation*}\begin{CD}
\Spec \calO @>>>\Spec S\\
@VVV @VV\psi V\\
\Spec R@>>\phi> X
\end{CD}\end{equation*}
commutes.

So, if $U=\Spec A$ is affine, then $U^\beth$ is nothing but the set of multiplicative seminorms $\vert.\vert_x$ on $A$ that are bounded, i.e. that fulfill $\vert f\vert_x\leq 1$ for all $f\in A$. The topology on $U^\beth$ is the one induced from $U^{an}$ and there is a natural anti-continuous \emph{reduction map} $r\mathrel{\mathop:}U^\beth\rightarrow U$ that sends $x\in U^\beth$ to the prime ideal $\big\{f\in A\big\vert \vert f\vert_x<1\big\}$. For a general scheme $X$ locally of finite type over $\mathbf{k}$\ChDan{, we again choose} a covering $U_i$ by open affine subsets, and now glue the $U_i^\beth$ over the closed subsets $r^{-1}(U_i\cap U_j)$ in order to obtain $X^\beth$. 

As an immediate application of the valuative criteria for separatedness and properness, we obtain that, if $X$ is separated, then $X^\beth$ is naturally a \ChMartin{locally closed} subspace of $X^{an}$, and, if $X$ is complete, then $X^\beth=X^{an}$.

\begin{example} The space $(\AA^1)^\beth$ is precisely the subspace of $(\AA^1)^{an}$ consisting of the points $\vert .\vert_{a,r}$ for $a\in \mathbf{k}$ and $r\in[0,1)$ as well as the Gauss point $\vert.\vert_{\infty,1}$.
\begin{center}\begin{tikzpicture}
\fill (8,0) circle (0.08 cm);
\fill (6,-2) circle (0.08 cm);
\fill (7,-2) circle (0.08 cm);
\fill (8,-2) circle (0.08 cm);
\fill (9,-2) circle (0.08 cm);
\fill (10,-2) circle (0.08 cm);

\draw (8,0) -- (6,-2);
\draw (8,0) -- (7,-2);
\draw (8,0) -- (8,-2);
\draw (8,0) -- (9,-2);
\draw (8,0) -- (10,-2);

\node at (8,-2.5) {$0$};
\node at (10.5,-2) {$\A^1$};
\node at (8.5,0) {$\eta$};
\node at (10,1) {$(\A^1)^\beth$};

\fill (6.3,-1) circle (0.03cm);
\fill (6.1,-1) circle (0.03cm);
\fill (5.9,-1) circle (0.03cm);
\fill (9.7,-1) circle (0.03cm);
\fill (9.9,-1) circle (0.03cm);
\fill (10.1,-1) circle (0.03cm);

\end{tikzpicture}\end{center}
\end{example}

\subsection{The case of toric varieties}\label{section_troptor}

Let $\mathbf{k}$ be a non-Archimedean field. As observed in \cite{Gubler_tropanalvar}, \cite{EinsiedlerKapranovLind_nonArchamoebas}, \cite{Gubler_guide}, and \cite{BakerPayneRabinoff_nonArchtrop}, there is an intricate relationship between non-Archimedean analytic geometry and tropical geometry. In particular, in many interesting situations the tropicalization of an algebraic variety $X$ over $\mathbf{k}$ can be regarded as a natural deformation retract of $X^{an}$, a so called \emph{skeleton} of $X^{an}$. In this section we are going to give a detailed explanation of this relationship in the simplest possible case, that of toric varieties.

Let $T$ a split algebraic torus with character lattice $M$ and cocharacter lattice $N$, and $X=X(\Delta)$ a $T$-toric variety that is defined by a rational polyhedral fan $\Delta$ in $N_\R$. We refer the reader to Section~\ref{Sec:toric} for a brief summary of this beautiful theory, and to \cite{Fulton_toric} and \cite{tv} for a more thorough account.

In \cite{Kajiwara_troptoric} and \cite{Payne_anallimittrop} Kajiwara and Payne independently construct a  \emph{tropicalization map}
\begin{equation*}
\trop_\Delta\mathrel{\mathop:}X^{an}\longrightarrow N_\R(\Delta)
\end{equation*} 
associated to $X$, whose codomain is a partial compactification of $N_\R$, uniquely determined by $\Delta$ (also see \cite[Section 4]{PopescuPampuStepanov_localtrop} and \cite[Section 3]{Rabinoff_newtonpolygon}). 

For a cone $\sigma$ in $\Delta$ set $N_\R(\sigma)=\Hom(S_\sigma,\Rbar)$, where $S_\sigma$ denotes the toric monoid $\sigma^\vee\cap M$ and write $\Rbar$ for the additive monoid $(\R\sqcup\{\infty\},+)$. Endow $N_\R(\sigma)$ with the topology of pointwise convergence. 

\begin{lemma}[\cite{Rabinoff_newtonpolygon} Proposition 3.4]\begin{enumerate}[(i)]
\item The space $N_\R(\sigma)$ has a stratification by locally closed subsets isomorphic to the vector spaces $N_\R/\Span(\tau)$ for all faces $\tau$ of $\sigma$. 
\item For a face $\tau$ of $\sigma$ the natural map $S_\sigma\rightarrow S_\tau$ induces the open embedding 
\begin{equation*}
N_\R(\tau)\hooklongrightarrow N_\R(\sigma)
\end{equation*} 
that identifies $N_\R(\tau)$ with the union of strata in $N_\R(\sigma)$ corresponding to faces of $\tau$ in $N_R(\sigma)$. 
\end{enumerate}\end{lemma}

So one can think of $N_\R(\sigma)$ as a partial compactification of $N_\R$ given by adding a vector space $N_\R/\Span(\tau)$ at infinity for every face $\tau\neq 0$ of $\sigma$. The partial compactification $N_\R(\Delta)$ of $N_\R$ is defined to be the colimit of the $N_\R(\sigma)$ for all cones $\sigma$ in $\Delta$. Since the stratifications on the $N_\R(\sigma)$ are compatible, the space $N_\R(\Delta)$ is a partial compactification of $N_\R$ that carries a stratification by locally closed subsets isomorphic to $N_\R/\Span(\sigma)$ for every cone $\sigma$ in $\Delta$. 

On the $T$-invariant open affine subset $X_\sigma=\Spec \mathbf{k}[S_\sigma]$ the \emph{tropicalization map}
\begin{equation*}
\trop_\sigma\mathrel{\mathop:}X_\sigma^{an}\longrightarrow N_\R(\sigma)
\end{equation*}
is defined by associating to an element $x\in X^{an}$ the homomorphism $s\mapsto -\log\vert\chi^s\vert_x$ in $N_\R(\sigma)=\Hom(S_\sigma,\Rbar)$. 

\begin{lemma}
The tropicalization map $\trop_\sigma$ is continuous and, for a face $\tau$ of $\sigma$, the natural diagram
\begin{equation*}\begin{CD}
X_\tau^{an}@>\trop_\tau>>N_\R(\tau)\\
@V\subseteq VV @VV\subseteq V\\
X_\sigma^{an}@>\trop_\sigma>> N_\R(\sigma)
\end{CD}\end{equation*}
commutes \ChJonathan{and is cartesian}.
\end{lemma}

Therefore we can glue the $\trop_\sigma$ on local $T$-invariant patches $X_\sigma$ and obtain a global continuous \emph{tropicalization map} 
\begin{equation*}
\trop_\Delta\mathrel{\mathop:}X^{an}\longrightarrow N_\R(\Delta) \ .
\end{equation*}
that restricts to $\trop_\sigma$ on $T$-invariant open affine subsets $X_\sigma$. Its restriction to a $T$-orbits is the usual tropicalization map in the sense of \cite[Section 3]{Gubler_guide}.

The following Proposition \ref{prop_toricskeleton} is well-known among experts and can be found in \cite[Section 2]{Thuillier} in the constant coefficient case, i.e. the case that $\mathbf{k}$ is trivially valued.

\begin{proposition}\label{prop_toricskeleton} The tropicalization map $\trop_\Delta\mathrel{\mathop:}X^{an}\rightarrow N_\R(\Delta)$ has a continuous section
$J_\Delta$ and the composition 
\begin{equation*}
\bfp_\Delta=J_\Delta\circ\trop_\Delta\mathrel{\mathop:}X^{an}\rightarrow X^{an}
\end{equation*}
defines a strong deformation retraction. 
\end{proposition}

The deformation retract 
\begin{equation*}
\frakS(X)=J_\Delta\big(N_\R(\Delta)\big)=\bfp_\Delta(X^{an})
\end{equation*}
is said to be the \emph{non-Archimedean skeleton} of $X^{an}$.

\begin{proof}[Proof sketch of Proposition \ref{prop_toricskeleton}]
Consider a $T$-invariant open affine subset $X_\sigma=\Spec \mathbf{k}[S_\sigma]$. We may construct the section $J_\sigma\mathrel{\mathop:}N_\R(\sigma)\rightarrow X_\sigma^{an}$ by associating to $u\in N_\R(\sigma)=\Hom(S_\sigma,\Rbar)$ the seminorm $J_\sigma(u)$ defined by 
\begin{equation*}
J_\sigma(u)(f)=\max_{s\in S_\sigma} \: \Bigl\{ \vert a_s\vert \: \exp\big(\!-\!u(s)\big) \Bigr\}
\end{equation*}
for $f=\sum_sa_s\chi^s\in \mathbf{k}[S_\sigma]$. A direct verification shows that $J_\sigma$ is continuous and fulfills $\trop_\sigma\circ J_\sigma=\id_{N_\R(\sigma)}$. 

The construction of $J_\sigma$ is compatible with restrictions to $T$-invariant affine open subsets and we obtain a global section $J_\Delta\mathrel{\mathop:}N_\R(\Delta)\rightarrow X^{an}$ of the tropicalization map $\trop_\Delta\mathrel{\mathop:}X^{an}\rightarrow N_\R(\Delta)$. 

Since $J_\Delta$ is a section of $\trop_\Delta$, the continuous map 
\begin{equation*}
\bfp_\Delta=J_\Delta\circ\trop_\Delta\mathrel{\mathop:}X^{an}\longrightarrow X^{an}
\end{equation*}
is a retraction map. On $X_\sigma$ the image $\bfp_\sigma(x)$ of $x\in X_\sigma^{an}$ is the seminorm given by 
\begin{equation*}
\bfp_\sigma(x)(f)=\max_{s\in S_\sigma} \: \bigl\{ \vert a_s\vert \: \vert \chi^s\vert_x \bigr\}
\end{equation*}
for $f=\sum_s a_s\chi^s\in \ChDan{\mathbf{k}}[S_\sigma]$. The arguments in \cite[Section 2.2]{Thuillier} generalize to this situation and show that $\bfp_\Delta$ is, in fact, a strong deformation retraction (see in particular \cite[Lemme 2.8(1)]{Thuillier}). 
\end{proof}

\begin{example}
The skeleton of $\AA^1$ is given by the half open line connecting $0$ to $\infty$, i.e., for trivially valued $\mathbf{k}$ we have 
\begin{equation*}
\frakS(X)=\Bigl\{\:\:\vert \cdot\vert_{0,r}\:\:\Big\vert\:\: r\in [0,1)\Bigr\}\cup\Bigl\{\:\: \vert\cdot\vert_{\infty,r}\:\:\Big\vert \:\: r\in[1,\infty)\Bigr\} 
\end{equation*}
in the notation of Example \ref{example_A1}.
\end{example}

Let $\mathbf{k}$ now be endowed with the trivial absolute value. As seen in \cite[Section 2.1]{Thuillier} the deformation retraction $\bfp_\Delta\mathrel{\mathop:}X^{an}\rightarrow X^{an}$ restricts to a deformation retraction $\bfp_X\mathrel{\mathop:}X^\beth\rightarrow X^\beth$, whose image is homeomorphic to the closure $\overline{\Delta}$ of $\Delta$ in $N_\R(\Delta)$. On $T$-invariant open affine subsets $X_\sigma=\Spec \mathbf{k}[S_\sigma]$ this homeomorphism is induced by the tropicalization map 
\begin{equation*}\begin{split}
\trop_X\mathrel{\mathop:}X^\beth&\longrightarrow \Hom(S_\sigma,\Rbar_{\geq 0})\\
x&\longrightarrow \big(s\mapsto -\log\vert\chi^s\vert_x\big)
\end{split}\end{equation*}
into the extended cone $\sigmabar=\Hom(S_\sigma,\Rbar_{\geq 0})$. 

\subsection{The case of logarithmic schemes}\label{section_troplog}
\subsubsection{Zariski logarithmic schemes}\label{section_troplog-Z}

Suppose that $\mathbf{k}$ is endowed with the trivial absolute value and let $X$ be logarithmically smooth over $\mathbf{k}$. In \cite{Thuillier} Thuillier constructs a strong deformation retraction $\mathbf{p}_X\mathrel{\mathop:}X^\beth\rightarrow X^\beth$ onto a closed subset $\frakS(X)$ of $X^\beth$, the \emph{skeleton} of $X^\beth$. We summarize the basic properties of this construction in the following Proposition \ref{prop_skeletonprop}. We refer to Definition \ref{definition_smalllogsch} for the notion of small Zariski logarithmic schemes.

\begin{proposition}\label{prop_skeletonprop}\begin{enumerate}[(i)]
\item The construction of $\bfp_X$ is functorial with respect to logarithmic morphisms, i.e. given a logarithmic morphism $f\mathrel{\mathop:}X\rightarrow Y$, there is a continuous map $f^\frakS\mathrel{\mathop:}\frakS(X)\rightarrow\frakS(Y)$ that makes the diagram
\begin{equation*}\begin{CD}
X^\beth @>\bfp_X>> \frakS(X)\\
@Vf^\beth VV @VVf^\frakS V\\
Y^\beth @>\bfp_Y>> \frakS(Y)
\end{CD}\end{equation*}
commute. \ChMartin{Moreover, if $f$ is logarithmically smooth, then $f^\frakS$ is the restriction of $f^\beth$ to $\frakS(X)$.}
\item For every strict \'etale neighborhood $U\rightarrow X$ that is small with respect to $x\in U$ and for every strict \'etale morphism $\gamma\mathrel{\mathop:}U\rightarrow Z$ into a $\gamma(x)$-small toric variety $Z$ the analytic map $\gamma^\beth$ induces a homeomorphism $\gamma^\frakS\mathrel{\mathop:}\frakS(U)\xrightarrow{\sim}\frakS(Z)$ that makes the diagram
\begin{equation*}\begin{CD}
U^\beth @>\bfp_U>> \frakS(U)\\
@V \gamma^\beth VV @V\sim V\gamma^\frakS  V\\
Z^\beth @>\bfp_Z>>\frakS(Z)
\end{CD}\end{equation*}
commute. 
\item The skeleton $\frakS(X)$ is the colimit of all skeletons $\frakS(U)$ associated to strict \'etale morphisms $U\rightarrow X$ from a Zariski logarithmic scheme $U$ that is small and the deformation retraction $\mathbf{p}_X\mathrel{\mathop:}X^\beth\rightarrow X^\beth$ is induced by the universal property of colimits. 
\end{enumerate}\end{proposition}

Suppose now that the logarithmic structure on $X$ is defined in the Zariski topology. In this case, following  \cite{Ulirsch_functroplogsch}, one can use the theory of Kato fans in order to define a tropicalization map $\trop_X\mathrel{\mathop:}X^\beth\rightarrow\Sigmabar_X$ generalizing the one of toric varieties. 

Let $F$ be a fine and saturated Kato fan and consider the \emph{cone complex} 
\begin{equation*}
\Sigma_F=F(\R_{\geq 0})=\Hom\big(\Spec\R_{\geq 0},F\big)
\end{equation*}
and the \emph{extended cone complex}
\begin{equation*}
\Sigmabar_F=F(\Rbar_{\geq 0})=\Hom\big(\Spec\Rbar_{\geq 0},F\big)\supseteq\Sigma_F
\end{equation*}
associated to $F$. In order to describe the structure of $\Sigma_F$ and $\Sigmabar_F$ one can use the \emph{structure map}
\begin{equation*}\begin{split}
\rho\mathrel{\mathop:}\Sigmabar_F&\longrightarrow F\\
u&\longmapsto u\big(\{\infty\}\big)
\end{split}\end{equation*}
and the \emph{reduction map}
\begin{equation*}\begin{split}
r\mathrel{\mathop:}\Sigmabar_F&\longrightarrow F\\
u&\longmapsto u(\Rbar_{>0}) \ .
\end{split}\end{equation*}

\begin{proposition}[\cite{Ulirsch_tropcomplogreg} Propostion 3.1]\label{prop_conecomplexreduction}
The inverse image $r^{-1}(U)$ of an open affine subset $U=\Spec P$ in $\Sigmabar_F$ is the canonical compactification $\sigmabar_U=\Hom(P,\Rbar_{\geq 0})$ of a rational polyhedral cone $\sigma_U=\Hom(P,\mathbb{R}_{\geq 0})$ and its relative interior $\mathring{\sigma}_U$ is given by $r^{-1}(x)$ for the unique closed point $x$ in $U$. 
\begin{enumerate}[(i)]
\item If $V\subseteq U$ for open affine subsets $U,V\subseteq F$, then $\sigma_V$ is a face of $\sigma_U$. 
\item For two open affine subsets $U$ and $V$ of $F$ the intersection $\sigma_U\cap\sigma_V$ is a union of finitely many common faces. 
\end{enumerate}
\end{proposition}

So the cone complex $\Sigma_F$ is a rational polyhedral cone complex in the sense of \cite{KKMS}. It naturally carries the weak topology, in which a subset $A\subseteq \Sigma_F$ is closed if and only if the intersections $A\cap \sigma_U$ for all open affine subsets $U=\Spec P$ of  $F$ are closed. The extended cone complex $\Sigmabar_F$ is a canonical compactification of $\Sigma_F$, carrying the weak topology with the topology of pointwise convergence on $\sigma_U=\Hom(P,\Rbar_{\geq 0})$ as local models. 

\begin{proposition}
\begin{enumerate}[(i)]
\item The reduction map $r\mathrel{\mathop:}\Sigmabar_F\rightarrow F$ is anti-continuous.
\item The structure map $\rho\mathrel{\mathop:}\Sigmabar_F\rightarrow F$ is continuous. 
\item There is a natural stratification 
\begin{equation*}
\bigsqcup_{x\in\Spec F}\rho^{-1}(x)\simeq\Sigmabar_F
\end{equation*}
of $\Sigmabar_F$ by locally closed subsets.
\end{enumerate}
\end{proposition}

Let $X$ be a Zariski logarithmic scheme that is logarithmically smooth over $\ChDan{\mathbf{k}}$ and denote by $\phi_X\mathrel{\mathop:}(X,\overline{\mathcal{O}}_X)\rightarrow F_X$ the \emph{characteristic morphism} into its Kato fan $F_X$. We write $\Sigma_X$ and $\overline{\Sigma}_X$ for the cone complex and the extended cone complex of $F_X$ respectively. 

Following \cite[Section 6.1]{Ulirsch_functroplogsch} one can define the \emph{tropicalization map} $\trop_X\rightarrow\overline{\Sigma}_X$ as follows: A point $x\in X^\beth$ can be represented by a morphism $\undernorm{x}\mathrel{\mathop:}\Spec R\rightarrow (X,\overline{\mathcal{O}}_X)$ for a valuation ring $R$ extending $\mathbf{k}$. Its image $\trop_X(x)$ in $\overline{\Sigma}_X=\Hom(\Spec\overline{\mathbb{R}}_{\geq 0},F_X)$ is defined to be the composition 
\begin{equation*}\begin{CD}
\Spec \overline{\mathbb{R}}_{\geq 0}@>\val^\#>>\Spec R @>\undernorm{x}>> (X,\overline{\mathcal{O}}_X)@>{\phi_X}>>F_X \ , 
\end{CD}\end{equation*}
where $\val^\#$ is the morphism $\Spec \overline{\mathbb{R}}_{\geq 0}\rightarrow\Spec R$ induced by the valuation $\val\mathrel{\mathop:}R\rightarrow\overline{\mathbb{R}}_{\geq 0}$ on $R$. 

\begin{proposition}[\cite{Ulirsch_functroplogsch} Proposition 6.2]\label{prop_tropprop}
\begin{enumerate}[(i)]
\item The tropicalization map is well-defined and continuous. It makes the diagrams
\begin{equation*}\begin{CD}
X^\beth @>\trop_X>>\overline{\Sigma}_X\\
@VrVV @VrVV\\
(X,\overline{\mathcal{O}}_X) @>\phi_X>> F_X
\end{CD}
\qquad
\begin{CD}
X^\beth @>\trop_X>>\overline{\Sigma}_X\\
@V\rho VV @V\rho VV\\
(X,\overline{\mathcal{O}}_X)@>\phi_X >>F_X
\end{CD}\end{equation*}
commute. 
\item A morphism $f\mathrel{\mathop:}X\rightarrow X'$ of Zariski logarithmic schemes, both logarithmically smooth and of finite type over $\ChDan{\mathbf{k}}$, induces a continuous map $\Sigmabar(f)\mathrel{\mathop:}\Sigmabar_X\rightarrow\Sigmabar_{X'}$ such that the diagram
\begin{equation*}\begin{CD}
X^\beth @>\trop_X >>\Sigmabar_X\\
@Vf^\beth VV @VV\Sigmabar(f)V\\
(X')^\beth @>\trop_{X'}>>\Sigmabar_{X'}
\end{CD}\end{equation*}
commutes. The association $f\mapsto\Sigmabar(f)$ is functorial in $f$. 
\end{enumerate}\end{proposition}

\begin{corollary}[Strata-cone correspondence, \cite{Ulirsch_tropcomplogreg} Corollary 3.5]\label{cor_stratacone}
There is an order-reversing one-to-one correspondence between the cones in $\Sigma_X$ and the strata of $X$. Explicitly it is given by 
\begin{equation*}
\mathring{\sigma}\longmapsto r\big(\trop_X^{-1}(\mathring{\sigma})\big)
\end{equation*}
for a a relatively open cone $\mathring{\sigma}\subseteq\Sigma_X$ and
\begin{equation*}
E\longmapsto \trop_X\big(r^{-1}(E)\cap X_0^{an}\big)
\end{equation*}
for a stratum $E$ of $X$. 
\end{corollary}

\begin{proof}
This is an immediate consequence of the commutativity of 
\begin{equation*}\begin{CD}
X^\beth @>\trop_X>>\overline{\Sigma}_X\\
@Vr VV @Vr VV\\
X@>\phi_X >>F_X 
\end{CD}\end{equation*} 
from Proposition \ref{prop_tropprop}, Proposition \ref{prop_conecomplexreduction} and the fact that $\phi_X$ sends every point in a stratum $E=E(\xi)$ to its generic point $\xi$. 
\end{proof}

\begin{corollary}[\cite{Ulirsch_tropcomplogreg} Corollary 3.6]
The tropicalization map induces a continuous map $X^\beth\cap X_0^{an}\rightarrow\Sigma_F$. 
\end{corollary}

\begin{proof}
This follows from the commutativity of 
\begin{equation*}\begin{CD}
X^\beth @>\trop_X>>\overline{\Sigma}_X\\
@V\rho VV @V\rho VV\\
(X,\overline{\mathcal{O}}_X)@>\phi_X >>F_X
\end{CD}\end{equation*}
and the observations that $\rho^{-1}(X_0)=X^\beth\cap X^{an}$ as well as $\rho^{-1}(X_0\cap F_X)=\Sigma_F$. 
\end{proof}

\subsubsection{\'Etale logarithmic schemes}\label{section_troplog-E}
Let $X$ be an \'etale logarithmic scheme that is logarithmically smooth over $\ChDan{\mathbf{k}}$. We can define the \emph{generalized extended cone complex} associated to $X$ as the colimit of all $\Sigmabar_{X'}$ taken over all strict \'etale morphisms $X'\rightarrow X$ from a Zariski logarithmic scheme $X'$. The tropicalization map $\trop_X\mathrel{\mathop:}X^\beth\rightarrow \Sigmabar_X$ is induced by the universal property of colimits. 

In analogy with Proposition \ref{prop_toricskeleton} we have the following compatibility result stating that $\bfp_X$ and $\trop_X$ are equal up to a natural homeomorphism. 

\begin{theorem}[\cite{Ulirsch_functroplogsch} Theorem 1.2]\label{thm_trop=skel}
Suppose that $X$ is logarithmically smooth over $\ChDan{\mathbf{k}}$. There is a natural homeomorphism $J_X\mathrel{\mathop:}\Sigmabar_X\xrightarrow{\sim}\frakS(X)$ making the diagram
\begin{center}\begin{tikzpicture}
  \matrix (m) [matrix of math nodes,row sep=2em,column sep=3em,minimum width=2em]
  {  
  & X^\beth & \\ 
 \frakS(X)  & & \Sigmabar_X  \\ 
  };
  \path[-stealth]
    (m-1-2) edge node [above left] {$\bfp_X$} (m-2-1)
    		edge node [above right] {$\trop_X$} (m-2-3)
    (m-2-3) edge node [below] {$J_X$} node [above] {$\sim$} (m-2-1);		
\end{tikzpicture}\end{center}
commute. 
\end{theorem}

Moreover, we obtain the following Corollary of Proposition \ref{prop_tropprop} (ii).

\begin{corollary}[\cite{Ulirsch_functroplogsch} Theorem 1.1]\label{cor_tropfunc}
A morphism $f\mathrel{\mathop:}X\rightarrow X'$ of logarithmic schemes, logarithmically smooth and of finite type over $\ChDan{\mathbf{k}}$, induces a continuous map $\Sigmabar(f)\mathrel{\mathop:}\Sigmabar_X\rightarrow\Sigmabar_{X'}$ such that the natural diagram
\begin{equation*}\begin{CD}
X^\beth @>\trop_X>>\overline{\Sigma}_X\\
@Vf^\beth VV @VV\Sigmabar(f) V\\
(X')^\beth @>\trop_{X'}>> \overline{\Sigma}_{X'} 
\end{CD}\end{equation*}
is commutative. The association $f\mapsto \Sigmabar(f)$ is functorial in $f$.
\end{corollary}

\section{Analytification of Artin fans}
\label{sec:anal}

\subsection{Analytification of Artin fans}

In Section \ref{section_troplog} we have seen that the extended cone complex $\Sigmabar_F$ associated to a Kato fan $F$ has topological properties analogous to the non-Archimedean analytic space $X^\beth$ associated to a scheme $X$ of finite type over $\ChDan{\mathbf{k}}$. Moreover, if $X$ is a Zariski logarithmic scheme that is logarithmically smooth over $k$, the tropicalization map $\trop_X\mathrel{\mathop:}X^\beth\rightarrow \Sigmabar_X$ is the ``analytification" of the characteristic morphism $\phi_X\mathrel{\mathop:}(X,\calO_X)\rightarrow F_X$. Using the theory of Artin fans we can make this analogy more precise, and even generalize the construction of $\trop_X$ to all logarithmic schemes. 

Let $\ChDan{\mathbf{k}}$ be endowed with the trivial absolute \ChDan{value}. As explained in \cite[Section V.3]{Ulirsch_Artinfanstrop} the $(.)^\beth$-functor, originally constructed in \cite{Thuillier}, generalizes to a pseudofunctor from the $2$-category of algebraic stacks locally of finite type over $\ChDan{\mathbf{k}}$ into the category of non-Archimedean analytic stacks, such that whenever $[U/R]\simeq \calX$ is a groupoid presentation of an algebraic stack $\calX$ we have an natural equivalence $[U^\beth/R^\beth]\simeq\calX^\beth$. We refer the reader to \cite{PortaYu_higherGAGA}, \cite{Ulirsch_nonArchstacks} , and  \cite[Section 6]{Yu_Gromovcompactness} for background on the theory of non-Archimedean analytic stacks. 

Let $\calX$ be an algebraic stack locally of finite type over $\ChDan{\mathbf{k}}$. Then the underlying topological space $\vert \calX^\beth\vert$ of the analytic stack $\calX$ can be identified with the set of equivalence \ChDan{classes} of pairs $(R,\phi)$ consisting of a valuation ring $R$ extending $\ChDan{\mathbf{k}}$ and a morphism $\phi\mathrel{\mathop:}\Spec R\rightarrow \calX$. Two such pairs $(\phi,R)$ and $(\psi,S)$ are said to be \emph{equivalent}, if there is a valuation ring $\calO$ extending both $R$ and $S$ such that the diagram 
\begin{equation*}\begin{CD}
\Spec \calO @>>>\Spec S\\
@VVV @VV\psi V\\
\Spec R@>>\phi> X
\end{CD}\end{equation*}
is $2$-commutative. The topology on $\vert\calX^\beth\vert$ is the coarsest making all maps $\vert U^\beth\vert\rightarrow\vert\calX^\beth\vert$ induced by surjective flat morphisms $U\rightarrow\calX$ from a scheme $U$ locally of finite type over $\ChDan{\mathbf{k}}$ onto $X$ into a topological quotient map. 

Suppose that $X$ is a logarithmic scheme that is logarithmically smooth and of finite type over $\ChDan{\mathbf{k}}$. Consider the natural strict morphism $X\rightarrow\calA_X$ into the Artin fan associated to $X$ as constructed in Section \ref{sec:artin}.

\begin{theorem}[\cite{Ulirsch_Artinfanstrop}]\label{thm_tropicalization=analytification}
There is a natural homeomorphism 
\begin{equation*}
\mu_X\mathrel{\mathop:}\big\vert\calA_X^\beth\big\vert\xrightarrow{\sim}\Sigmabar_X
\end{equation*} 
that makes the diagram
\begin{center}\begin{tikzpicture}
  \matrix (m) [matrix of math nodes,row sep=2em,column sep=3em,minimum width=2em]
  {  
  & X^\beth & \\ 
  \big\vert\calA_X^\beth\big\vert  & & \Sigmabar_X  \\ 
  };
  \path[-stealth]
    (m-1-2) edge node [above left] {$\phi_X^\beth$} (m-2-1)
    		edge node [above right] {$\trop_X$} (m-2-3)
    (m-2-1) edge node [below] {$\mu_X$} node [above] {$\sim$} (m-2-3);		
\end{tikzpicture}\end{center}
commute. 
\end{theorem}

So, by applying the functor $(.)^\beth$ to the morphism $X\rightarrow \calA_X$ we obtain the tropicalization map on the underlying topological spaces. Note that this construction also works for \'etale logarithmic schemes and we do not have to take colimits as in Section \ref{section_troplog-E} (they are already taken in the construction of $\cA_X$). Theorem \ref{thm_trop=skel} immediately yields the following Corollary.

\begin{corollary}\label{cor_skeleton=analytification}
There is a natural homeomorphism 
\begin{equation*}
\tilde{\mu}_X\mathrel{\mathop:}\big\vert\calA_X^\beth\big\vert\xrightarrow{\sim}\frakS(X)
\end{equation*}
that makes the diagram
\begin{center}\begin{tikzpicture}
  \matrix (m) [matrix of math nodes,row sep=2em,column sep=3em,minimum width=2em]
  {  
  & X^\beth & \\ 
  \big\vert\calA_X^\beth\big\vert  & & \frakS(X)  \\ 
  };
  \path[-stealth]
    (m-1-2) edge node [above left] {$\phi_X^\beth$} (m-2-1)
    		edge node [above right] {$\bfp_X$} (m-2-3)
    (m-2-1) edge node [below] {$\tilde{\mu}_X$} node [above] {$\sim$} (m-2-3);		
\end{tikzpicture}\end{center}
commute. 
\end{corollary}

\subsection{Stack quotients and tropicalization}

Let $X$ be a $T$-toric variety over $\ChDan{\mathbf{k}}$. In this case, Theorem \ref{thm_tropicalization=analytification} precisely says that on the underlying topological spaces the tropicalization map $\trop_X\mathrel{\mathop:}X^\beth\rightarrow \Sigmabar_X$ is nothing but the analytic stack quotient map $X^\beth\rightarrow [X^\beth/T^\beth]$. Due to the favorable algebro-geometric properties of toric varieties we can generalize this interpretation to general ground fields. 

Let $\ChDan{\mathbf{k}}$ be any non-Archimedean field\ChQile{,} and suppose that $X$ is a $T$-toric variety defined by a rational polyhedral fan $\Delta\subseteq N_\R$ as in Section \ref{section_troptor}. Denote by $T^\circ$ the non-Archimedean analytic subgroup 
\begin{equation*}
\big\{x\in T^{an}\:\big\vert\:\vert \chi^m\vert_x=1 \textrm{ for all }m\in M\big\}
\end{equation*}
of the analytic torus $T^{an}$. Note that, if $\ChDan{\mathbf{k}}$ is endowed with the trivial absolute value, then $T^\circ=T^\beth$. In general, we can think of $T^\circ$ as a non-Archimedean analogue of the real $n$-torus $N_{S^1}=N\otimes_\ZZ S^1$ naturally sitting in $N_{\CC^\ast}=N\otimes_\ZZ\CC^\ast$. The $T$-operation on $X$ induces an operation of $T^\circ$ on $X^{an}$.

\begin{theorem}[\cite{Ulirsch_nonArchstacks} Theorem 1.1]\label{thm_trop=quot}
There is a natural homeomorphism 
\begin{equation*}
\mu_X\mathrel{\mathop:}\big\vert[X^{an}/T^\circ]\big\vert\xrightarrow{\sim}N_\R(\Delta)
\end{equation*}
that makes the diagram
\begin{center}\begin{tikzpicture}
  \matrix (m) [matrix of math nodes,row sep=2em,column sep=3em,minimum width=2em]
  {  
  & X^{an} & \\ 
  \big\vert[X^{an}/T^\circ]\big\vert  & & N_\R(\Delta)  \\ 
  };
  \path[-stealth]
    (m-1-2) edge (m-2-1)
    		edge node [above right] {$\trop_X$} (m-2-3)
    (m-2-1) edge node [below] {$\mu_X$} node [above] {$\sim$} (m-2-3);		
\end{tikzpicture}\end{center}
commute. 
\end{theorem}

The proof Theorem \ref{thm_trop=quot} is based on establishing that the skeleton $\frakS(X)$ of $X$ is equal to the set of $T^\circ$-invariant points of $X^{an}$. Then the statement follows, since by \cite[Proposition 5.4 (ii)]{Ulirsch_nonArchstacks} the topological space $\big\vert[X^{an}/T^\circ]\big\vert$ is the colimit of the maps 
\begin{equation*}
T^\circ\times X^{an}\rightrightarrows X^{an} \ .
\end{equation*}

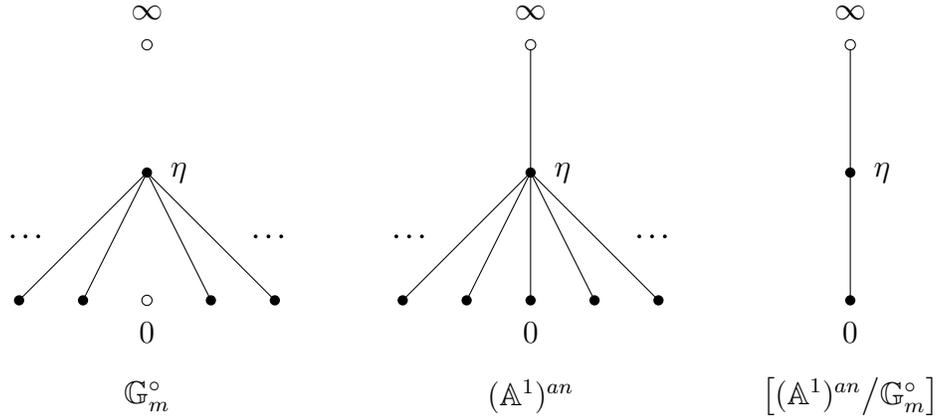
\begin{figure}[H]
\begin{center}
\begin{tikzpicture}[scale=.85]
\draw (6,2) circle (0.08 cm);
\fill (6,0) circle (0.08 cm);
\fill (4,-2) circle (0.08 cm);
\fill (5,-2) circle (0.08 cm);
\fill (6,-2) circle (0.08 cm);
\fill (7,-2) circle (0.08 cm);
\fill (8,-2) circle (0.08 cm);

\draw (6,0) -- (6,1.92);
\draw (6,0) -- (4,-2);
\draw (6,0) -- (5,-2);
\draw (6,0) -- (6,-2);
\draw (6,0) -- (7,-2);
\draw (6,0) -- (8,-2);

\node at (6,-2.5) {$0$};
\node at (6,2.5) {$\infty$};
\node at (6.5,0) {$\eta$};
\node at (6,-3.5) {$(\A^1)^{an}$};

\fill (4.3,-1) circle (0.03cm);
\fill (4.1,-1) circle (0.03cm);
\fill (3.9,-1) circle (0.03cm);
\fill (7.7,-1) circle (0.03cm);
\fill (7.9,-1) circle (0.03cm);
\fill (8.1,-1) circle (0.03cm);

\draw (0,2) circle (0.08 cm);
\fill (0,0) circle (0.08 cm);
\fill (-2,-2) circle (0.08 cm);
\fill (-1,-2) circle (0.08 cm);
\draw (0,-2) circle (0.08 cm);
\fill (1,-2) circle (0.08 cm);
\fill (2,-2) circle (0.08 cm);

\draw (0,0) -- (-2,-2);
\draw (0,0) -- (-1,-2);
\draw (0,0) -- (1,-2);
\draw (0,0) -- (2,-2);

\node at (0,-2.5) {$0$};
\node at (0,2.5) {$\infty$};
\node at (0.5,0) {$\eta$};
\node at (0,-3.5) {$\G_m^{\circ}$};

\fill (-1.7,-1) circle (0.03cm);
\fill (-1.9,-1) circle (0.03cm);
\fill (-2.1,-1) circle (0.03cm);
\fill (1.7,-1) circle (0.03cm);
\fill (1.9,-1) circle (0.03cm);
\fill (2.1,-1) circle (0.03cm);

\fill (11,0) circle (0.08cm);
\fill (11,-2) circle (0.08cm);
\draw (11,2) circle (0.08cm);

\draw (11,0) -- (11,-2);
\draw (11,0) -- (11,1.92);

\node at (11.5,0) {$\eta$};
\node at (11,-3.5) {$\big[(\A^1)^{an}\big/\G_m^\circ\big]$};
\node at (11,-2.5) {$0$};
\node at (11,2.5) {$\infty$};

\end{tikzpicture}
\end{center}
\caption{Consider the affine line $\A^1$ over a trivially valued field $\ChDan{\mathbf{k}}$. The non-Archimedean unit circle $\G_m^\circ$ is given as the subset of elements in $x\in(\A^1)^{an}$ with $\vert t\vert_x=1$, where $t$ denotes a coordinate on $\A^1$. The skeleton $\frakS(\A^1)$ of $(\A^1)^{an}$ is the line connecting $0$ to $\infty$. It is precisely the set of ``$\G_m^\circ$-invariant'' points in $(\A^1)^{an}$\ChQile{,} and therefore naturally homeomorphic to the topological space underlying $\big[(\A^1)^{an}\big/\G_m^\circ\big]$ (see \cite[Example 6.2]{Ulirsch_nonArchstacks}).
}
\end{figure}

\section{Where we are, where we want to go}
\label{sec:future}

\subsection{Skeletons fans and tropicalization over non-trivially valued fields}

In almost all of our discussion, we have constructed Kato fans, Artin fans and skeletons for a logarithmic variety over a trivially valued field $k$. One exception is the discussion of toric varieties in Section \ref{section_troptor}. This is quite useful and important: given a subvariety $Y$ of a toric variety $X$, over a non-trivially valued field, the tropicalization $\trop(Y)$ of $Y$ is a polyhedral subcomplex of $N_\RR(\Delta)$ which is not itself a fan. Much of the impact of tropical geometry relies on the way $\trop(Y)$ reflects on the geometry of $Y$. So even though the toric variety $X$ itself can be defined over a field with trivial valuation, the fact that $Y$ - whose field of definition is non-trivially valued - has a combinatorial shadow is fundamental.

Can we define Artin fans over non-trivially valued fields? What should their structure be?   In what generality can we canonically associate a skeleton to  a logarithmic structure? 

Some results in this directions are available in the recent paper of Gubler, Rabinoff and Werner \cite{GRW}. See also Werner's contribution to this volume \cite{Werner}. 

\subsection{Improved fans and moduli spaces}

One of the primary applications of tropical geometry, and therefore of fans, is through moduli spaces. Mikhalkin's correspondence theorem \cite{Mikhalkin}, Nishinou and Siebert's vast extension \cite{NS}, the work \cite{CMR} of Cavalieri, Markwig and Ranganathan, and the manuscript \cite{ACGS}, all show that tropical moduli spaces $M^{\trop}(X^{\trop})$ of tropical curves in tropical varieties serve as good approximations of the tropicalization $\trop(M(X))$ of moduli spaces $M(X)$ of algebraic curves in algebraic varieties. But the picture is not perfect: 

\begin{enumerate} 
\item \ChDan{Apart from very basic cases of the moduli space of curves itself \cite{ACP} and genus-0 maps with toric targets \cite{Ranganathan-toric}} the tropical moduli space does not  coincide with the tropicalization of the algebraic moduli space $$M^{\trop}(X^{\trop}) \ \ \neq\ \ \ \trop(M(X)).$$
\item Even when it does, it is always a coarse moduli space, lacking the full power of universal families.
\end{enumerate}

These seem to be two distinct challenges, but experience shows that they are closely intertwined. It also seems that the following problem is part of the puzzle:
\begin{enumerate}
\item[(3)] The construction to the Artin fan $\cA_X$ of a logarithmic scheme $X$ is not  functorial for all morphisms of logarithmic schemes.
\end{enumerate}

The following program might inspire one to go some distance towards these challenges:

\begin{enumerate}
\item Good combinatorial moduli:
\begin{enumerate}
\item Construct a satisfactory monoidal theory of cone spaces and cone stacks as described in Sections \ref{sec:cone-space-motivation} and \ref{sec:cone-stack-motivation}.
\item Construct a corresponding enhancement of $M_{g,n}^{\trop}$ which is a fine moduli stack of tropical curves, with a universal family.
\end{enumerate}
This is part of current work of Cavalieri, Chan, Ulirsch and Wise.
\item Stacky Artin fans: 
\begin{enumerate}
\item Find a way to relax the representability condition in the definition of $\cA_X$ so that the enhanced moduli spaces above are canonically associated to the moduli stacks $M_{g,n}$
\item Try to extend all of the above to moduli of maps.
\end{enumerate}
\end{enumerate}

\bibliographystyle{amsalpha}
\bibliography{logsurvey}

\end{document}